\documentclass[11pt,a4paper]{article}

\usepackage{amsmath,amsthm,amsfonts,a4wide}
\usepackage {latexsym}
\usepackage{amssymb}
\usepackage[dvips]{epsfig}

\renewcommand{\thesection}{\arabic{section}}

\newtheorem{theorem}{Theorem}[section]
\newtheorem{lemma}[theorem]{Lemma}
\newtheorem{prop}[theorem]{Proposition}
\newtheorem{defi}[theorem]{Definition}
\newtheorem{remark}[theorem]{Remark}

\newtheorem{corollary}[theorem]{Corollary}
\newcommand{\prf}{{\noindent{\it Proof.\quad }}}

\renewcommand{\theequation}{\thesection .\arabic{equation}}
\let\subs\subsection
\renewcommand\subsection{\setcounter{equation}{0}
\gdef\theequation{\thesubsection \arabic{equation}}\subs}
\let\sect\section
\renewcommand\section{\setcounter{equation}{0}
\gdef\theequation{\thesection .\arabic{equation}}\sect}

\newcommand{\gange}{\cdot}

\newcommand{\dps}{\displaystyle}

\newcommand{\e}{\mathrm{e}} 
\newcommand{\I}{\mathrm{i}} 

\renewcommand{\bar}[1]{\ensuremath{\overline{#1}}}

\newcommand{\supp}{\ensuremath{\textrm{supp}}}

\title{Cyclic $p$-roots of prime length $p$ and related\\
complex Hadamard matrices}

\author{\textsc{Uffe Haagerup}}
\date{}

\begin{document}
\maketitle

\begin{abstract}
In this paper it is proved, that for every prime number $p$, the set of cyclic
$p$-roots in $\mathbb{C}^p$ is finite. Moreover the number of
cyclic $p$-roots counted with multiplicity is equal to
$\left(\begin{smallmatrix}2p-2\\ p-1\end{smallmatrix}\right)$.
In particular, the number of complex circulant Hadamard matrices
of size $p$, with diagonal entries equal to 1, is less than or equal
to $\left(\begin{smallmatrix}2p-2\\ p-1\end{smallmatrix}\right)$.
\end{abstract}

\section{Introduction}
In $\cite{bj}$, G\"oran Bj\"orck introduced the cyclic $n$-roots for
every $n\in\mathbb{N}\ (n\ge 2)$ as the solutions 
$z=(z_0,\ldots,z_{n-1})\in\mathbb{C}^n$ to the following $n$
polynomial equations:
\begin{equation}\label{e1.1}
\begin{aligned}
z_0+z_1+\ldots+z_{n-1} &= 0\\
z_0z_1+z_1z_2+\ldots+z_{n-1}z_0 &= 0\\
& \vdots\\
z_0z_1\gange\ldots\gange 
z_{n-2}+\ldots+z_{n-1}z_0\gange\ldots\gange z_{n-3} &= 0\\
z_0z_1\gange\ldots\gange z_{n-1} &= 1
\end{aligned}
\end{equation}
This system of equations is invariant under cyclic permutation 
of the indices $(0,1,\ldots,n-1)$.\\ The motivation for studying
the system of equations (\ref{e1.1}) was to study
bi-unimodular sequences of length $n$, i.e. elements
$(x_0,x_1,\ldots,x_{n-1})$ in $\mathbb{C}^n$ for which
\[
|x_j|=1\quad\textrm{and}\quad |\hat{x}_j|=
1\quad\textrm{for}\quad 0\le j\le n-1
\]
where $\hat{x}=(\hat{x}_0,\hat{x}_1,\ldots,\hat{x}_{n-1})$ is
the Fourier Transformed of $x$ w.r.t. the group 
$\mathbb{Z}_n=\mathbb{Z}/n\mathbb{Z}$, i.e.
\begin{equation}\label{1.2}
\hat{x}_j = \frac{1}{\sqrt{n}}\sum_{k=0}^{n-1}\e^{\I 2\pi jk/n}x_k,
\quad 0\le j\le n-1.
\end{equation}
If $x=(x_0,\ldots,x_{n-1})\in\mathbb{C}^n$ and $|x_j|=1$,
$1\le j\le n$, then by $\cite{bj}$, $x$ is a biunimodular sequence
if and only if
\[
(z_0,\ldots,z_{n-1})=\Big(\frac{x_1}{x_0},\frac{x_2}{x_1},\ldots,
\frac{x_{n-1}}{x_{n-2}},\frac{x_0}{x_{n-1}}\Big)
\]
is a cyclic $n$-root, and this gives a one-to-one correspondence
between bimodular sequences $(x_0,x_1,\ldots,x_{n-1})$ with $x_0=1$
and cyclic $n$-roots of modulus $1$.

A complex Hadamard matrix of size $n$ is a matrix
\[
H=(h_{jk})_{j,k=0,\ldots,n-1}
\]
for which all entries are complex numbers with modulus $1$, and
\[
H^*H = nI.
\]
Moreover $H$ is called circulant, if the entries $h_{jk}$ only
depend on $j-k$ (calculated modulo $n$). By \cite{bs} a 
$n\times n$ matrix $H$ is a complex circulant Hadamard matrix
if and only if
\[
h_{jk}=x_{j-k},\quad j,k\in\{0,\ldots,n-1\}
\]
for a biunimodular sequence $x=(x_0,\ldots,x_{n-1})$ (again,
indices must be calculated modulo $n$). Hence there is also a
one-to-one correspondence between complex cyclic $n$ roots and
circulant Hadamard
matrices of size $n$ with diagonal entries equal to 1.

It is elementary to solve the cyclic $n$-root problem
(\ref{e1.1}) for $n=2,3$ and $4$. In 1991-92 Bj\"orck and
Fr\"oberg found all cyclic $n$-roots for $5\le n\le 8$ by computer
algebra methods (cf. \cite{bf1} and \cite{bf2}), for the case $n=7$ see also \cite{baf}. Moreover in 2001 Faug\`ere found
all cyclic $9$-roots by developing more advanced software for
computer algebra (cf. \cite{fa}). For $2\le n\le 9$, the total
number $\gamma(n)$ of cyclic $n$-roots and the number
$\gamma_u(n)$ of cyclic $n$-roots of modulus 1 are given by the table:
\begin{center}
\begin{tabular}{|c||c|c|c|c|c|c|c|c|}
\hline
$n$ & 2 & 3 & 4 & 5 & 6 & 7 & 8 & 9\\
\hline
$\gamma(n)$ & 2 & 6 & $\infty$ & 70 & 156 & 924 & $\infty$ & $\infty$\\
\hline
$\gamma_u(n)$ & 2 & 6 & $\infty$ & 20 & 48 & 532 & $\infty$ & $\infty$\\
\hline
\end{tabular}
\end{center}
For further results on cyclic $n$-roots and circulant Hadamard
matrices, see also \cite{ha}.

Based on the values of $\gamma(n)$ for $n=2,3,5$ and 7. Ralf Fr\"oberg
conjectured that
$\gamma(p)=\left(\begin{smallmatrix}2p-2\\ p-1\end{smallmatrix}\right)$
for all prime numbers $p$. In this paper we will prove, that for
every prime number $p$, the number of cyclic $p$-roots counted with
multiplicity is equal to 
$\left(\begin{smallmatrix}2p-2\\ p-1\end{smallmatrix}\right)$.
For $p=2,3,5$ and 7 all the cyclic $p$-roots have multiplicity $1$,
but we do not know, whether this holds for all primes. In the
non-prime case $n=9$, Faug\`ere found isolated cyclic $9$-roots with
multiplicity $4$ (cf. \cite{fa}).

Let us next outline the main steps in our proof. In section 2 we
prove that there is a one-to-one correspondence between solutions
to (\ref{e1.1}) and solutions to the following
system of $2n-2$ equations in $2n-2$ variables
$(x_1,\ldots,x_{n-1},y_1,\ldots,y_{n-1})$,
\begin{equation}\label{1.3}
\left.\begin{array}{ll}
x_jy_j=1, & 1\le j\le n-1\\
\hat{x}_j\hat{y}_{-j} = 1, & 1\le j\le n-1
\end{array}\right.
\end{equation}
where $x=(1,x_1,\ldots,x_{n-1})$, $y=(1,y_1,\ldots,y_{n-1})$ and
$\hat{x},\hat{y}$ are the Fourier transformed vectors of $x$ and
$y$ as defined by (\ref{1.2}).

In section 3, we prove that for every prime number $p$, the set of
solutions to (\ref{1.3}) with $n=p$ is a finite set. The proof is based on a
Theorem of  Chebotar\"ev from 1926, which asserts, that when $p$ is
a prime number, then all square sub-matrices of the matrix
\[
(\e^{\I 2\pi jk/p})_{j,k=0,\ldots,p-1}
\]
are non-singular. Having only finitely many solutions to (\ref{1.3}) 
the same holds for (\ref{e1.1}), but in order to count the
number of solutions in (\ref{1.3}) and (\ref{e1.1}), we have
in section 4 collected a number of (mostly) well known results on
multiplicity of proper holomorphic functions
$\varphi\!: U\rightarrow V$, where $U,V$ are regions in 
$\mathbb{C}^n$, and on multiplicity of the isolated zeros of such a
function. The main result needed is that for all $w\in V$, the
number of solutions to $\varphi(z)=w$ (i.e. the number of zeros of
$\varphi_w\!: z\rightarrow \varphi(z)-w$) counted with multiplicity
is equal to the multiplicity of $\varphi$, and it is therefore
independent of $w\in V$ (cf. Theorem \ref{t4.8}). Using this we
can count the number of solutions to (\ref{1.3}) with multiplicity,
by counting instead the solutions
$(x_1,\ldots,x_{p-1},y_1,\ldots,y_{p-1})\in\mathbb{C}^{2p-2}$
to
\begin{equation}\label{1.4}
\left.\begin{array}{ll}
x_jy_j=0, & 1\le j\le p-1\\
\hat{x}_j\hat{y}_{-j}=0, & 1\le j\le p-1
\end{array}\right.
\end{equation}
where $x=(1,x_{1},\ldots,x_{p-1})$ and $y=(1,y_1,\ldots,y_{p-1})$
as in (\ref{1.3}). The latter problem can be solved by linear
algebra (cf. section 5) and it has exactly
$\left(\begin{smallmatrix}2p-2\\ p-1\end{smallmatrix}\right)$
solutions all with multiplicity $1$. Hence (\ref{1.3}) has
$\left(\begin{smallmatrix}2p-2\\ p-1\end{smallmatrix}\right)$
solutions counted with multiplicity.

It is clear from section 2, that (\ref{e1.1}) and 
(\ref{1.3}) has the same number of distinct solutions. In section 6,
we prove that the same also holds when solutions are counted
according to their multiplicities. This is not obvious, because,
when passing from (\ref{1.3}) to (\ref{e1.1}) the number
of variables is changed twice in the process, first from $2p-2$
to $p-1$ and next from $p-1$ to $p$.

In section 7, we use the methods from the previous sections to count the number of cyclic $p$-roots of simple index $k$, where $k\in\mathbb{N}$ divides $p-1$. Following \cite{bj} and \cite{bh} a cyclic $p$-root has simple index $k$, if the corresponding cyclic $p$-root on $x$-level is constant on the cosets of the unique index $k$ subgroup of $(\mathbb{Z}_p^*,\cdot)$. The cyclic $p$-roots of simple index $k$ can be determined by solving the following set of equations in $k$ variables $c_0,c_1,\ldots,c_{k-1}\in\mathbb{C}^*$:
\begin{equation}\label{1.5}
c_a + \frac{1}{c_{a+m}}+\sum_{i,j=0}^{k-1}n_{ij}\frac{c_{j+a}}{c_{i+a}}=0\quad (0\le a\le k-1)
\end{equation}
where $m$ and $n_{ij}$ are certain integers depending on $p$ and $k$ (cf. \cite{bj} and section 7 of this paper for more details). For $k=1,2,3$ all cyclic $p$-roots of simple index $k$ has been explicitly computed in \cite{bj} and \cite{bh}. The number of distinct cyclic $p$-roots of simple index $k$ is 2 (resp. 6, 20) for $k=1$ (resp. 2, 3) for all primes for which $k$ divides $p-1$. We prove in Theorem 7.1 that the number of solutions to (\ref{1.5}) counted with multiplicity is equal to $\left(\begin{smallmatrix}2k\\ k\end{smallmatrix}\right)$ for all $k\in\mathbb{N}$ and all primes for which $k$ divides $p-1$.

\subsection*{Acknowledgement}
I wish to thank G\"oran Bj\"orck for many fruitful discussions
on cyclic $n$-roots, since we first met in 1992, and for constantly
encouraging me to write up a detailed proof of the main result of this
paper after a very preliminary version of the proof was communicated to
him in the summer of 1996. I also wish to thank Bahman Saffari for
his interest in this result and for giving me the opportunity to 
present it at the workshop on Harmonic Analysis and Number Theory
at CIRM/Luminy, October 2005.

\section{Reformulations of the cyclic $n$-root problem}\label{reformation}
Recall that the cyclic $n$-roots are the solutions
$z=(z_0,z_1,\ldots,z_{n-1})\in\mathbb{C}^n$ to the system of equations:
\begin{equation}\label{2.1}
\begin{aligned}
z_0+z_1+\ldots+z_{n+1} &= 0\\
z_0z_1+z_1z_2+\ldots+z_{n-1}z_0 &= 0\\
& \vdots\\
z_0z_1\gange\ldots\gange z_{n-2}+\ldots+
z_{n-1}z_0\gange\ldots\gange z_{n-3} &= 0\\
z_0z_1\gange\ldots\gange z_{n-1} &= 1
\end{aligned}
\end{equation}
Note that by the last equation $z_i\in\mathbb{C}^* =
\mathbb{C}\setminus \{0\}$ for every cyclic $n$-root
$z=(z_0,\ldots,z_{n-1})$. Let $z\in(\mathbb{C}^*)^n$ be a cyclic
$n$-root, and define $x=(x_0,\ldots,x_{n-1})\in(\mathbb{C}^*)^n$ by
\begin{equation}\label{2.2}
x_0 = 1,\ x_1 = z_0,\ x_2 = z_0z_1,\ \ldots,\ x_{n-1} =
z_0z_1\cdot 
\ldots\cdot z_{n-2}
\end{equation}
Then clearly
\[
\frac{x_{j+1}}{x_j}=z_j, \quad j=0,1,\ldots,n-2
\]
and by the last equation in (\ref{2.1}) the same formula also 
holds for $j=n-1$. Moreover, by the first $n-1$ equations in
(\ref{2.1}), $x=(x_0,\ldots,x_{n-1})$ is a solution to
\begin{equation}\label{2.3}
\begin{aligned}
x_0 &= 1\\
\frac{x_1}{x_0}+\frac{x_2}{x_1}+\ldots+\frac{x_0}{x_{n-1}} &= 0\\
\frac{x_2}{x_0}+\frac{x_3}{x_1}+\ldots+\frac{x_1}{x_{n-1}} &= 0\\
&\vdots\\
\frac{x_{n-1}}{x_0}+\frac{x_0}{x_1}+\ldots+\frac{x_{n-2}}{x_{n-1}} &= 0
\end{aligned}
\end{equation}
Conversely if $x=(x_0,\ldots,x_{n-1})\in(\mathbb{C}^*)^n$ is a 
solution to (\ref{2.3}), then
\[
(z_0,z_1,\ldots,z_{n-1})=
\Big(\frac{x_1}{x_0},\frac{x_2}{x_1},\ldots,\frac{x_0}{x_{n-1}}\Big)
\]
is a solution to (\ref{2.1}). \emph{We will call the 
solutions to (\ref{2.3}) cyclic $n$-roots on $x$-level}.\\
\\
Instead of imposing the condition $x_0=1$, it would be equivalent
to look for solutions to the last $n-1$ equations of
(\ref{2.3}) in the subset $(\mathbb{C}^*)^n/\!\!\sim$ of
the complex projective space
$P_{n-1}=(\mathbb{C}^n\setminus\{0\})/\!\!\sim$, where
$x,x'\in\mathbb{C}^n\setminus\{0\}$ are equivalent
$(x\sim x')$ iff $x'=cx$ for some $c\in\mathbb{C}^*$.\\
\\
Suppose $x=(x_0,\ldots,x_{n-1})\in(\mathbb{C}^*)^n$ is a solution
to (\ref{2.3}), and put $y_j=\frac{1}{x_j}$, $j=0,\ldots,{n-1}$.
Then
\[
(x,y)=(x_0,\ldots,x_{n-1},y_0,\ldots,
y_{n-1})\in\mathbb{C}^n\times\mathbb{C}^n
\]
is a solution to
\begin{equation}\label{2.4}
\begin{aligned}
x_0=y_0 &= 1\\
x_ky_k &= 1\textrm{ for } 1\le k\le n-1\\
\sum_{m=0}^{n-1}x_{k+m}y_m &= 0\textrm{ for } 1\le k\le n-1
\end{aligned}
\end{equation}
where again all indices are counted modulo $n$. Conversely if 
$(x,y)\in\mathbb{C}^n\times\mathbb{C}^n$ is a solution to
(\ref{2.4}), then $x\in(\mathbb{C}^*)^n$ and $x$ is a solution to
(\ref{2.3}), because $x_ny_n=1$ for $0\le k\le n-1$.
\emph{We will call the solutions
$(x_0,\ldots,x_{n-1},y_0,\ldots,
y_{n-1})\in\mathbb{C}^n\times\mathbb{C}^n$ to (\ref{2.4}) cyclic
$n$-roots on $(x,y)$-level.}\\
\\
Instead of imposing the conditions $x_0=y_0=1$, it would be
equivalent to look for solutions to
\begin{equation}\label{2.5}
\begin{aligned}
x_ky_k &= x_0y_0, \quad 1\le k\le n-1\\
\sum_{m=0}^{n-1}x_{k+m}y_m &= 0,\quad 1\le k\le n-1
\end{aligned}
\end{equation}
in the subset $(\mathbb{C}^*)^n/\!\!\sim \times\;
(\mathbb{C}^*)^n/\!\!\sim$ of $P_{n-1}\times P_{n-1}$.

\begin{lemma}\label{l2.1} Let $n,v\in\mathbb{C}^n$ and
let $\hat{u},\hat{v}\in\mathbb{C}^n$ be the transformed vectors, i.e.
\[
\hat{u}=Fu,\quad \hat{v}=Fv
\]
where $F$ is the unitary matrix
\[
F=\frac{1}{\sqrt{n}}\Big(\e^{\I 2\pi jk/n}\Big)_{j,k=0,\ldots,n-1}
\]
Still calculating indices cyclic modulo $n$, we have
\begin{alignat}{2}
\label{2.6}\hat{u}_j\hat{v}_{-j} & =
\frac{1}{n}\sum_{k=0}^{n-1}\e^{\I 2\pi jk/n}
\bigg(\sum_{m=0}^{n-1}u_{k+m}v_m\bigg), && \quad 0\le j\le n-1\\
\label{2.7}\sum_{j=0}^{n-1}\e^{-\I 2\pi kj/n} \hat{u}_j\hat{v}_j & =
\sum_{m=0}^{n-1} u_{k+m}v_m, && \quad 0\le k\le n-1
\intertext{In particular}
\label{2.8}\sum_{j=0}^{n-1}\hat{u}_j\hat{v}_{-j} &=
\sum_{m=0}^{n-1}u_mv_m
\end{alignat}
\end{lemma}

\begin{proof} Let $0\le j\le n-1$. Then
\[
\hat{u}_j\hat{v}_{-j} =
\frac{1}{n}\sum_{l,m=0}^{n-1}\e^{\I 2\pi j(l-m)/n} u_kv_m.
\]
Hence, if we replace $(l,m)$ with $(k+m,m)$ in the double sum, we get
\begin{align*}
\hat{u}_j\hat{v}_{-j} &= 
\frac{1}{n}\sum_{k,m=0}^{n-1}\e^{\I 2\pi jk/n}u_{k+m}v_m\\
&= \frac{1}{n}\sum_{k=0}{n-1}\e^{\I 2\pi jk/n}
\bigg(\sum_{m=0}^{n-1}u_{k+m}v_m\bigg),
\end{align*}
which proves (\ref{2.6}). Note that (\ref{2.6}) can also be written as
\[
\Big(\hat{u}_j\hat{v}_{-j}\Big)_{j=0}^{n-1} =
\frac{1}{\sqrt{n}}F\Bigg(\bigg(\sum_{m=0}^{n-1}u_{k+m}v_m\bigg)_{k=0}^{n-1}\Bigg).
\]
Since $F$ is unitary and symmetric, $F^{-1}=\overline{F}$
(complex conjugation). Thus
\[
\sqrt{n}\; \bar{F}\bigg(\Big(\hat{u}_j\hat{u}_{-j}\Big)_{j=0}^{n-1}\bigg) =
\bigg(\sum_{m=0}^{n-1}u_{k+m}v_m\bigg)_{k=0}^{n-1}
\]
which proves (\ref{2.7}). (\ref{2.8}) is the special case
$k=0$ of (\ref{2.7}). Note that (\ref{2.8}) can also be proved
by applying Parseval's formula
\[
\sum_{j=0}^{n-1}\hat{u}_j\bar{\hat{w}}_j=\sum_{m=0}^{n-1}u_m\bar{w}_m
\]
to $w=\bar{v}$.
\end{proof}

\begin{prop}\label{p2.2}
\noindent The equations (\ref{2.4}) for cyclic $n$-roots 
on $(x,y)$-level are equivalent to the following set of
equations for $(x,y)\in\mathbb{C}^n\times\mathbb{C}^n$.
\begin{equation}\label{2.9}
\begin{aligned}
x_0 = y_0 &= 1 && \\
x_ky_k &= 1, && 1\le k\le n-1\\
\hat{x}_k\hat{y}_{-k} &= 1, && 1\le k\le n-1 
\end{aligned}
\end{equation}
where $\hat{x}=Fx$ and $\hat{y}=Fy$ as in lemma \ref{l2.1}.
\end{prop}

\begin{proof}
Assume $(x,y)$ is a solution to (\ref{2.4}). By the last $n-1$
equations of (\ref{2.4}),
\[
\sum_{m=0}^{n-1}x_{k+m}y_m = 0\qquad 1\le k\le n-1
\]
and by the first $n+1$ equations of (\ref{2.4})
\[
\sum_{m=0}^{n-1}x_my_m=n
\]
Hence, by (\ref{2.6})
\begin{align*}
\hat{x}_j\hat{y}_{-j} &= 
\frac{1}{n}\sum_{k=0}^{n-1}\e^{\I 2\pi jk/n}
\bigg(\sum_{m=0}^{n-1}u_{k+m}v_m\bigg)\\
&= \frac{1}{n}(n+0+\ldots+0)\\
&= 1
\end{align*}
for $j=0,\ldots,n-1$. Hence (\ref{2.4}) implies (\ref{2.9}). 
Conversely if $(x,y)\in\mathbb{C}^n\times\mathbb{C}^n$
satisfies (\ref{2.9}), then
\[
\hat{x}_j\hat{y}_{-j} = 1\quad\textrm{for}\quad 1\le j\le k
\]
By (\ref{2.8}) and the first $n+1$ equations of (\ref{2.9}),
\[
\sum_{j=0}^{n-1} \hat{x}_j\hat{y}_{-j}=\sum_{m=0}^{n-1}x_my_m=n
\]
and therefore
\[
\hat{x}_0\hat{y}_0=n-\sum_{j=1}^{n-1}\hat{x}_j\hat{y}_{-j} = n-(n-1) = 1
\]
Thus by (\ref{2.7})
\begin{align*}
\sum_{k=0}^{n-1}x_{k+m}y_m &=
\sum_{j=0}^{n-1}\e^{-\I 2\pi kj/n}\hat{u}_j\hat{v}_{-j}\\
&= \sum_{j=0}^{n-1}\e^{-\I 2\pi kj/n}\\
&= 0\quad\textrm{for}\quad 1\le k\le n-1
\end{align*}
Hence (\ref{2.9}) implies (\ref{2.4}).
\end{proof}
For later use, (cf. proof of Corollary 5.4.)
we prove the following extension of Proposition 2.2. 

\begin{prop}\label{p2.3}
\noindent Let $a_1,\ldots,a_{n-1},c_1,\ldots,c_{n-1}\in\mathbb{C}$.
Then for $(x,y)\in\mathbb{C}^n\times\mathbb{C}^n$, the set of equations
\begin{equation}\label{2.10}
\begin{aligned}
x_0=y_0 &= 1\\
x_ky_k &= a_k,\quad 1\le k\le n-1\\
\sum_{m=0}^{n-1}x_{k+m}y_m &= c_k, \quad 1\le k\le n-1
\end{aligned}
\end{equation}
is equivalent to
\begin{equation}\label{2.11}
\begin{aligned}
x_0=y_0 &=1\\
x_ky_k &= a_k,\quad 1\le k\le n-1\\
\hat{x}_k\hat{y}_{-k} &= b_k,\quad 1\le k\le n-1
\end{aligned}
\end{equation}
where
\begin{equation}\label{2.12}
b_j=\frac{1}{n}\bigg(1+\sum_{m=1}^{n-1}a_m+
\sum_{k=1}^{n-1}\e^{\I 2\pi jk/n}c_k\bigg),\quad 1\le j\le n-1
\end{equation}
Moreover for fixed $a_1,\ldots,a_{n-1},b_1,\ldots,b_{n-1}\in\mathbb{C}$, 
the $n-1$ equations (\ref{2.12}) have a unique solution
$(c_1,\ldots,c_{n-1})\in\mathbb{C}^{n-1})$ given by
\begin{equation}\label{2.13}
c_k=1+\sum_{m=1}^{n-1}a_m+\sum_{j=1}^{n-1}(\e^{-\I 2\pi kj/n}-1)b_j,
\quad 1\le k\le n-1.
\end{equation}
\end{prop}

\begin{proof}
Assume (\ref{2.10}). Then
\[
\sum_{m=0}^{n-1}x_my_m=1+\sum_{m=1}^{n-1}a_m.
\]
Hence by (\ref{2.6}) and the last $(n-1)$ equations in (\ref{2.10})
\begin{align*}
\hat{x}_j\hat{y}_{-j} &= \frac{1}{n}
\bigg(\sum_{m=1}^{n-1}x_my_m + \sum_{k=1}^{n-1}\e^{\I 2\pi jk/n}
\bigg(\sum_{m=0}^{n-1}x_{k+m}y_m\bigg)\bigg)\\
&= \frac{1}{n}
\bigg(1+\sum_{m=1}^{n-1}a_m+\sum_{k=1}^{n-1}\e^{\I 2\pi jk/n}c_k\bigg)
\end{align*}
for $0\le j\le n-1$. Hence (\ref{2.10}) implies (\ref{2.11}), 
with $b_1,\ldots,b_{n-1}$ as in (\ref{2.12}).\\
\\
We next show that (\ref{2.12}) implies (\ref{2.13}). Put
\[
b_0=\frac{1}{n}\bigg(1+\sum_{m=1}^{n-1}a_m+\sum_{k=1}^{n-1}c_k\bigg).
\]
Then (\ref{2.12}) holds for $0\le j\le n-1$. Hence, if we furthermore put
\[
c_0=1+\sum_{m=1}^{n-1}a_m,
\]
then
\[
b_j=\frac{1}{n}\sum_{k=0}^{n-1}\e^{\I 2\pi jk/n}c_k,\quad j=0,\ldots,n-1.
\]
Hence, by Fourier inversion, we have
\begin{equation}\label{2.14}
c_k=\sum_{j=0}^{n-1}\e^{-\I 2\pi jk/n}b_j,\quad k=0,\ldots,n-1
\end{equation}
In particular
\[
1+\sum_{m=1}^{n-1}a_m=c_0=b_0+\sum_{j=1}^{n-1}b_j.
\]
Therefore
\[
b_0=1+\sum_{m=1}^{n-1}a_m-\sum_{j=1}^{n-1}b_j
\]
which inserted in (\ref{2.14}) gives
\begin{align*}
c_k &= b_0+\sum_{j=1}^{n-1}\e^{-\I 2\pi jk/n}b_j\\
&= 1+\sum_{m=1}^{n-1}a_m+\sum_{j=1}^{n-1}(\e^{-\I 2\pi jk/n}-1)b_j,
\quad 0\le k\le n-1
\end{align*}
which proves (\ref{2.13}).\\
\\
Finally, we show that (\ref{2.11}) implies (\ref{2.10}), when
(\ref{2.12}) holds (or equivalently (\ref{2.13}) holds).
Assume $(x,y)\in\mathbb{C}^n\times\mathbb{C}^n$ satisfies
(\ref{2.11}) for given $a_1,\ldots,a_{n-1},
b_1,\ldots,b_{n-1}\in\mathbb{C}$. By (\ref{2.11}) and (\ref{2.8})
we get
\[
1+\sum_{m=1}^{n-1}a_m = \sum_{m=0}^{n-1}x_my_m = 
\sum_{j=0}^{n-1}\hat{x}_j\hat{y}_{-j}=
\hat{x}_0\hat{y}_0+\sum_{j=1}^{n-1}bj.
\]
Therefore
\begin{equation}\label{2.15}
\hat{x}_0\hat{y}_0 = 1+\sum_{m=1}^{n-1}a_m-\sum_{j=1}^{n-1}b_j
\end{equation}
Hence by (\ref{2.7}) we have for $0\le k\le n-1$,
\begin{align*}
\sum_{m=0}^{n-1}x_{k+m}y_m &=
\sum_{j=0}^{n-1}\e^{-\I 2\pi kj/b}\hat{x}_j\hat{y}_{-j}\\
&= \hat{x}_0\hat{y}_0+\sum_{j=0}^{n-1}\e^{-\I 2\pi kj/n}bj\\
&= 1+\sum_{m=1}^{n-1}a_m + \sum_{j=0}^{n-1}(\e^{-\I 2\pi kj/n}-1)b_j.
\end{align*}
Thus (\ref{2.11}) implies (\ref{2.10}) with $c_1,\ldots,c_{n-1}$
given by (\ref{2.13}).
\end{proof}

\section{Finiteness of the set of cyclic $p$-roots of prime length $p$}
We shall use the following two classical results:

\begin{theorem}\label{thm3.1}
A compact algebraic variety in $\mathbb{C}^n$ is a finite set.\\
\end{theorem}

\prf This is well known, see e.g. [Ru, Thm 14.3.i].

\begin{theorem}\label{thm3.2} (Chebotar\"ev, 1926).
Let $p$ be a prime number and let $F_p$ denote the unitary matrix
of Fourier transform on $\mathbb{C}^p$:
\[
F_p=\Big(\frac{1}{\sqrt{p}}\e^{\I 2\pi kl/p}\Big)_{k,l=0,\ldots,p-1}.
\]
Then for every two finite subsets $K,L\subseteq\{0,\ldots,p-1\}$ of
the same size $|K|=|L|\ge 1$, the corresponding submatrix
\[
(F_p)_{K\times L} =
\Big(\frac{1}{\sqrt{p}}\e^{\I 2\pi kl/p}\Big)_{k\in K,l\in L}
\]
has non-zero determinant.\\
\end{theorem}

\prf  See [SL, p. 29-30] and references given there.
\medskip

The following application of Chebotar\"ev's Theorem has been known
to the author since 1996. After the results of this paper were
presented at CIRM in October 2005, we learned, that it has been
proved independently by Terence Tao (cf. [Ta, Thm 1.1]). In the
same paper, Tao also presents a short and selfcontained proof of
Chebotar\"ev's theorem.

\begin{prop}\label{p3.3}
Let $u=(u_0,\ldots,u_{p-1})\in\mathbb{C}^p$ and let
$\hat{u}=F_pu$ be the Fourier transformed vector. If $u\ne 0$, then
\begin{equation}\label{3.1}
|\supp(u)|+|\supp(\hat{u})|\ge p+1
\end{equation}
where for $z\in\mathbb{C}^p$, $|\supp(z)|$ denotes the number of $\;i\in\{0, 1,\ldots, p-1\}$ for which $z_i\ne 0$.
\end{prop}

\begin{proof}\
Let $p$ be a prime number, let $u\in\mathbb{C}^p\setminus\{0\}$,
assume that
\[
|\supp(u)|+|\supp(\hat{u})|\le p.
\]
Put $L=\supp(u)$ and note that $L\ne\emptyset$. Moreover
\[
|\mathbb{Z}_p\setminus\supp(\hat{u})| =
p-|\supp(\hat{u})|\ge |\supp(u)|=|L|.
\]
Hence, we can choose $K\subseteq\mathbb{Z}_p\setminus\supp(\hat{u})$,
such that $|K|=|L|$. For every $k\in K$
\begin{equation}\label{3.2}
\frac{1}{\sqrt{p}}\sum_{l\in L}\e^{\I 2\pi kl/n}u_l = \hat{u}_k = 0.
\end{equation}
By Chebotar\"ev's Theorem (Theorem 3.2), the matrix
\[
\Big(\frac{1}{\sqrt{p}}\e^{\I 2\pi kl/n}\Big)_{k\in K,l\in L}
\]
has non-zero determinant. Hence by (\ref{3.2}) 
$u_l=0$ for all $l\in L=\supp(\hat{u})$, which implies that
$u=0$ and we have reached a contradictim. Therefore (\ref{3.1})
holds for every $u\in\mathbb{C}^p\setminus\{0\}$.
\end{proof}

\begin{lemma}\label{l3.4}
Let $n\in\mathbb{N}$. If the number of solutions
$(x,y)\in\mathbb{C}^n\times\mathbb{C}^n$ to (\ref{2.9}) is
infinite, then there exists $u,v\in\mathbb{C}^n\setminus\{0\}$, such that
\[
u_kv_k=0\quad\textrm{and}\quad \hat{u}_k\hat{v}_{-k} = 0
\]
for $k=0,1,\ldots,n-1$.
\end{lemma}

\begin{proof}\ 
Let $W\subseteq\mathbb{C}^n\times\mathbb{C}^n$ be the set of
solutions to the $2n$ polynomial equations (\ref{2.9}) and assume that $W$ have infinite many elements. Since 
$W$ is an algebraic variety, we get by Theorem 3.1 and the
Heine-Borel Theorem, that $W$ is an unbounded set. Put
\[
\| z\|_{2} = \bigg(\sum_{j=0}^{n-1}|z_j|^2\bigg)^\frac{1}{2},
\quad z\in\mathbb{C}^n.
\]
We choose a sequence of elements $(x^{(m)},y^{(m)})$ in $W$,
$(m\in\mathbb{N})$ such that
\begin{equation}\label{3.3}
\lim_{n\rightarrow\infty}
\big(\| x^{(m)}\|_2^2 + \| y^{(m)}\|_2^2\big)^\frac{1}{2}=+\infty.
\end{equation}
Put next
\[
u^{(m)} = \frac{1}{\| x^{(m)}\|_2}x^{(m)},\quad v^{(m)}=
\frac{1}{\| y^{(m)}\|_2}y^{(m)}.
\]
Then $\|u^{(m)}\|_2=\|v^{(m)}\|_2=1$, i.e.
$(u^{(m)},v^{(m)})\in S^{2n-1}\times S^{2n-1}$ 
where $S^{2n-1}$ denotes the unit sphere in $\mathbb{C}^n$. 
Since $S^{2n-1}\times S^{2n-1}$ is compact, we can by passing to
a subsequence assume that
\[
\lim_{m\rightarrow\infty}(u^{(m)},v^{(m)})=(u,v)
\]
for some $u,v\in S^{2n-1}$. Since
$x,y\in {W}$, $x_0^{(m)}=y_0^{(m)}=1$ for all
$m\in\mathbb{N}$. Therefore
\[
\| x^{(m)}\|_2^2 = 1+c_m,\quad \|y^{(m)}\|_2^2 = 1+d_m
\]
for some non-negative real numbers $c_m$, $d_m$. Thus
\[
\| x^{(m)}\|_2^2\| y^{(m)}\|_2^2 = (1+c_m)(1+d_m) \ge 1+c_m+d_m
= \| x^{(m)} \|_2^2 + \| y^{(m)}\|_2^2 - 1.
\]
Hence by (\ref{3.3}),
\begin{equation}\label{3.4}
\lim_{n\rightarrow\infty}\| x^{(m)}\|_2\|y^{(m)}\|_2 = +\infty.
\end{equation}
Since $(x^{(m)},y^{(m)})$ satisfies (\ref{2.9}) for all $m$,
we have for $1\le k\le n-1$
\[
x_k^{(m)}y_k^{(m)} = 1,\quad \widehat{x_k^{(m)}}\widehat{y_{-k}^{(m)}} = 1
\]
and the same equalities holds for $k=0$, by (\ref{2.8}) combined
with $x_0^{(m)}=y_0^{(m)}=1$. Therefore
\[
u_kv_k=\hat{u}_k\hat{v}_{-k}=
\lim_{m\rightarrow\infty}\big(\|x^{(m)}\|_2\|y^{(m)}\|_2\big)^{-1}=0
\]
for $0\le k\le n-1$, which proves lemma 3.4.
\end{proof}

\begin{theorem}\label{thm3.5}
Let $p$ be a prime number, then the set of cyclic $p$-roots is finite.
\end{theorem}

\begin{proof}
The transformations of the cyclic $n$-root problem in section 2
from (\ref{2.1}) to (\ref{2.3}) and later from
(\ref{2.3}) to (\ref{2.4}) and (\ref{2.9}) do not change the
number of distinct solutions. Therefore it is sufficient to show,
that the set of solutions $W$ to (\ref{2.9}) is finite in the
case $n=p$.

Assume $|W|=+\infty$. Then by lemma 3.5 there exist 
$u,v\in\mathbb{C}^p\setminus\{0\}$, such that
\[
u_kv_k=0\quad\textrm{and}\quad \hat{u}_k\hat{v}_{-k}=0
\]
for $k=0,1,\ldots,p-1$, i.e.
\begin{equation*}
\supp(u) \cap \supp(v) = \emptyset\quad {\text and}\quad
\supp(\hat{u}) \cap (-\supp(\hat{v})) = \emptyset
\end{equation*}
Hence
\begin{equation*}
|\supp(u)| + |\supp(v)| \le p\quad {\text and}\quad
|\supp(\hat{u})| + |\supp(\hat{v})| \le p
\end{equation*}
and therefore
\begin{equation}\label{3.5}
|\supp(u)|+|\supp(\hat{u})|+|\supp(v)|+|\supp(\hat{v})|\le 2p.
\end{equation}
However, by Proposition 3.3 the left hand side of (\ref{3.5}) is
larger or equal to $2(p+1)$. This gives a contradiction, and we
have therefore proved, that the set $W$ of solutions to
(\ref{2.9}) is finite.
\end{proof}

\section{Multiplicity of a proper holomorphic function}
Let $U$, $V$ be regions in $\mathbb{C}^n$ (i.e. $U$ and $V$ are 
non-empty connected open subsets of $\mathbb{C}^n$). A holomorphic
function $\varphi: U\rightarrow V$ is called \emph{proper} if for
every compact subset $K$ and $V$,
$\varphi^{-1}(K)=\{z\in U\ |\ \varphi(z)\in K\}$ is a compact
subset of $U$. When $\varphi$ is proper its Jacobian
$J(z)=\det(\varphi'(z))$ can not vanish for all $z\in U$
(cf. [Ru1, 15.1.3]). Following [Ru1, 15.1.4], we let $M$ denote
the set
\[
M=\{z\in U\ |\ J(z)=0\}.
\]
Its range $\varphi(M)\subseteq V$ is called the set of
\emph{critical values} for $\varphi$ and $V\setminus\varphi(M)$
is called the set of \emph{regular values} for $\varphi$. By
[Ru, Prop. 15.1.5 and Thm. 15.1.9] we have

\begin{theorem}\label{thm4.1}
Let $U$, $V$ be regions in $\mathbb{C}^n$ and let 
$\varphi: U\rightarrow V$ be a proper holomorphic function and let
$\varphi(M)$ be the set of critical values for $\varphi$, then
\begin{itemize}
\item[(a)] $\varphi(U)=V$.
\item[(b)] The set $V\setminus\varphi{(M)}$ of regular values 
for $\varphi$ is a connected, open and dense subset of $V$.
\item[(c)] There is a unique natural number $m\in N$ (called
the multiplicity of $\varphi$) such that the number of elements
$|\varphi^{-1}(w)|$ in $\varphi^{-1}(w)$ satisfies
\begin{align*}
|\varphi^{-1}(w)| &= m\quad\textrm{for}\quad w\in V\setminus\varphi(M)\\
|\varphi^{-1}(w)| &< m\quad\textrm{for}\quad w\in\varphi(M).
\end{align*}
\item[(d)] The critical set $\varphi(M)$ is a zero-variety in $V$,
i.e. $\varphi(M)=\{w\in V\ |\ h(w)=0\}$ for some holomorphic
function $h: U\rightarrow\mathbb{C}$.
\end{itemize}
\end{theorem}

\begin{remark} 
\emph{The set of critical values $\varphi(M)$ is a zero set with respect to the
$2n$-dimensional Lebesgue measure $m_{2n}$ in
$\mathbb{C}^n\approx\mathbb{R}^{2n}$, i.e. $m_{2n}(\varphi(M))=0$.
This follows from Sard's Theorem (cf. [AY, Theorem 0.11]).}
\end{remark}

\begin{prop} 
{\text{\rm{[AY, Chap 1, Prop. 2.1]}}}:
Let $U,V$ be regions in $\mathbb{C}^n$
and let $\varphi: U\rightarrow V$ be a holomorphic function. Let
$a\in U$ be an isolated zero for $\varphi$, and choose a
neighborhood $U_a$ of $a$, such that $\varphi(z)\ne 0$ when
$z\in U_a\setminus\{ a\}$. Then there exists an $\varepsilon > 0$
such that for Lebesgue almost all $w\in B(0,\varepsilon)$, the function
\begin{equation}\label{4.2}
\varphi_w(z) = \varphi(z)-w
\end{equation}
has only simple zeros in $U_a$ (i.e, the Jacobian
$\det(\varphi'_w)$ does not vanish at the zeros of $\varphi_w)$,
and their number depends neither on $w$ nor on the choice of
the neighborhood $U_a$.
\end{prop}

\begin{defi}
The number of zeros to (4.1) indicated in Prop. 4.3 is called
the multiplicity of the isolated zero $a$ for $\varphi$.
\end{defi}
\medskip

An isolated zero $a$ for $\varphi$ has multiplicity one if and
only if $\det(\varphi'(a))\ne 0$ (cf. [AY, Chap 1, Prop 2.2 and
Prop 2.3]).

\begin{remark}
\emph{The multiplicity defined above is also called the 
\emph{geometric multiplicity} of an isolated zero
(cf [Ts, p. 16]). It coincides with the
\emph{algebraic multiplicity} of $a$:}
\[
\mu_a(\varphi) = \dim(O_a/I_a(\varphi)),
\]
\emph{where $O_a$ is the ring of holomorphic germs at $a$, and
$I_a(\varphi)$ is  the ideal in $O_a$ generated by the $n$
coordinate functions of $\varphi$ (cf. [Ts, p. 148]).}

\emph{We will also use the $n$-dimensional version of Rouch\'es Theorem
(cf. [AY, Thm. 2.5 and remark after Thm. 2.5]).}
\end{remark}

\begin{theorem}
Let $U$, $V$ be regions in $\mathbb{C}^n$ and let $D$ be a
bounded open set, such that $\bar{D}\subseteq U$ and $\partial D$
is piecewise smooth. Let $f,g: U\rightarrow V$ be holomorphic
functions, such that
\[
\forall z\in\partial D\quad \forall t\in [0,1]:\qquad f(z)+tg(z)\ne 0.
\]
Then $f$ and $f+g$ have only isolated zeros in $D$, and the two
functions $f$ and $f+g$ have the same number of zeros in $D$ counted
with multiplicity.
\end{theorem}

\begin{defi}
Let $U,V$ be regions in $\mathbb{C}^n$, $\varphi: U\rightarrow V$
be a proper holomorphic function, and let $w\in V$. By the number
$m(w)$ of solutions $z\in U$ to $\varphi(z)=w$ counted with
multiplicity, we mean the number of zeros of
$\varphi_w(z)=\varphi(z)-w$ in $U$ connected with multiplicity.
\end{defi}

The following theorem is probably well known but since we have
not found a concrete reference to it in the literature, we
include a proof.

\begin{theorem}\label{t4.8}
Let $U,V$ be regions in $\mathbb{C}^n$ and let
$\varphi: U\rightarrow V$ be a proper holomorphic function of
multiplicity $m$ (as defined in Theorem 4.1(c)). Then for every
$w\in V$, the number $m(w)$ of solutions $z\in U$ to $\varphi(z)=w$
counted with multiplicity is equal to $m$.
\end{theorem}

\begin{proof}
Let $\varphi(M)$ denote the set of critical values for $\varphi$
as in Theorem 4.1. For $w\in V$ we put
\[
\varphi_w(z)=\varphi(z)-w,\quad z\in U.
\]
Note that the Jacobian $J_w(z)=\det(\varphi'_w(z))$ is equal to
the Jacobian of $\varphi$. Assume first, that
$w\in V\setminus\varphi(M)$. Then the Jacobian of $\varphi_w$
is non-zero at all the zeros of $\varphi_w$ and hence all
the zeros have multiplicity 1. Hence by Theorem 4.1,
\[
m(w)=|\varphi^{-1}(w)|=m,\quad w\in V\setminus\varphi(M).
\]
Let now $w\in V$ be arbitrary. Choose an $\varepsilon > 0$
such that $\overline{B(w,\varepsilon)}$ is contained in $V$.
By the properness of $\varphi$,
\[
K=\varphi^{-1}(\overline{B(w,\varepsilon)})
\]
is a compact subset of $U$. Moreover
\begin{equation}\label{4.7}
|\varphi(z)-w| >\varepsilon\quad\textrm{for}\quad z\in U\setminus K
\end{equation}
and since $\partial K\subseteq \overline{U\setminus K}$, we have
\begin{equation}\label{4.8}
|\varphi(z)-w|\ge \varepsilon\quad\textrm{for}\quad z\in \partial K.
\end{equation}
Let $v\in B(w,\varepsilon)$. Then
$\varphi_v=\varphi_w+c$ where $c=w-v\in\mathbb{C}^n$, and 
by (\ref{4.8})
\[
|c|<\varepsilon\le |\varphi_w(z)|,\quad z\in\partial K.
\]
Assume first that the boundary $\partial K$ of $K$ is piecewise
smooth. Then we can apply Theorem 4.6 to $f=\varphi_w$ and
$g=c$, and obtain, that $\varphi_w$ and $\varphi_v$ have the
same number of zeros (counted with multiplicity) in
$\overset{\circ}{K}=K\setminus\partial K$. By (\ref{4.7}) and
(\ref{4.8}), neither $\varphi_w$ nor $\varphi_v=\varphi_w+c$,
$|c|<\varepsilon$ has zeros in $\partial K$ or $U\setminus K$. Hence
\[
m(v) = m(w),\quad v\in B(w,\varepsilon).
\]
Since $V\setminus\varphi(M)$ is dense in $V$ by Theorem 4.1, we
can choose a $v\in B(w,\varepsilon)\setminus\varphi(M)$ and for
this $v$, $m(w)=m(v)=m$ by the first part of the proof.

If $\partial K$ is not piecewise smooth, one can find a compact set $K'$ with piecewise smooth boundary, such that $K\subseteq K'\subseteq U$, for instance $K'$ can be a polyhedron or a finite
union of disjoint polyhedrons. Then the proof of $m(w)=m$ can
be completed as above by using $K'$ instead of $K$.
\end{proof}

\section{The number of cyclic $p$-roots on $(x,y)$-level}
Throughout this section $p$ is a prime number. We will show,
that for $n=p$, the numbers of solutions to (\ref{2.4}) and
(\ref{2.9}) counted with multiplicity are both equal to
$\left(\begin{smallmatrix}2p-2\\p-1\end{smallmatrix}\right)$.
In both cases we will consider $x_0,y_0$ as the fixed numbers
$x_0=y_0=1$, so the problems (\ref{2.4}) and (\ref{2.9}) have
$2p-2$ variables: $x_1,\ldots,x_{p-1},y_1,\ldots,y_{p-1}$.

\begin{lemma}\label{l5.1}
Let $x',y'\in\mathbb{C}^{p-1}$, $x'=(x_1,\ldots,x_{p-1})$,
$y'=(y_1,\ldots,y_{p-1})$, put
\[
x=(1,x_1,\ldots,x_{p-1}),\qquad y=(1,y_1,\ldots,y_{p-1}),
\]
and let $\hat{x}=F_p x$, $\hat{y}=F_p y$
be their Fourier transformed vectors in $\mathbb{C}^p$.
Consider the function
$\varphi:\mathbb{C}^{2p-2}\rightarrow\mathbb{C}^{2p-2}$ given
by the coordinate functions
\begin{align}
\varphi_j(x',y') &= x_jy_j,\quad 1\le j\le p-1\label{5.1}\\
\varphi_{p-1+j}(x',y') &=
\hat{x}_j\hat{y}_{-j},\quad 1\le j\le p-1\label{5.2}.
\end{align}
Then $\varphi$ is a proper holomorphic function.
\end{lemma}

\begin{proof}
Clearly $\varphi$ is a holomorphic function of $\mathbb{C}^{2p-2}$
into $\mathbb{C}^{2p-2}$. For $R>0$, we put
\[
\overline{B}(0,R)=\{w\in\mathbb{C}^{2p-2}\ |\ \|w\|_2\le R\}.
\]
Assume that $\varphi$ is not proper. Then for some $R>0$,
$\varphi^{-1}(\overline{B}(0,R))$ is not a bounded subset of
$\mathbb{C}^{2p-2}$. Hence there exists a sequence
$(z^{(m)})_{m=1}^\infty$ in $\mathbb{C}^{2p-2}$ such that
\[
\lim_{m\rightarrow\infty}\| z^{(m)}\|_2=\infty
\]
while
\begin{equation}\label{5.3}
\|\varphi(z^{(m)})\|_2\le R,\quad m\in\mathbb{N}.
\end{equation}
Write $z^{(m)}=(x_1^{(m)},\ldots,x_{p-1}^{(m)},y_1^{(m)},
\ldots,y_{p-1}^{(m)})$ and put 
\[
x^{(m)}=(1,x_1^{(m)},
\ldots,x_{p-1}^{(m)}),\ \ y^{(m)}=(1,y_1^{(m)},
\ldots,y_{p-1}^{(m)}).
\]
Then
\[
\|x^{(m)}\|_2^2\|y^{(m)}\|_2^2 =
\bigg(1+\sum_{j=1}^{p-1}|x_j^{(m)}|^2\bigg)
\bigg(1+\sum_{j=1}^{p-1}|y_j^{(m)}|^2\bigg) \ge 1+\|z^{(m)}\|_2^2.
\]
Hence
\begin{equation}\label{5.4}
\lim_{n\rightarrow\infty} \|x^{(m)}\|_2\|y^{(m)}\|_2=\infty.
\end{equation}
The rest of the proof will follow the proof of Lemma 3.4 and 
Theorem 3.5. By passing to a subsequence, we can obtain, that
the sequences
\[
u^{(m)}=\frac{1}{\|x^{(m)}\|_2} x^{(m)},\quad
v^{(m)}=\frac{1}{\|y^{(m)}\|_2} y^{(m)}
\]
both converge in the unit sphere $S^{2p-1}$ of $\mathbb{C}^p$. Put
\[
u=\lim_{m\rightarrow\infty}u^{(m)},
\quad v=\lim_{m\rightarrow\infty}v^{(m)}.
\]
By (\ref{5.1}), (\ref{5.2}) and (\ref{5.3}),
\[
|x_j^{(m)}y_j^{(m)}|\le R\quad\textrm{and}\quad 
|\widehat{x_j^{(m)}}\widehat{y_{-j}^{(m)}}|\le R
\]
for $1\le j\le p-1$. Hence by (\ref{5.4})
\[
u_jv_j=\lim_{m\rightarrow\infty}u_j^{(m)}v_j^{(m)}=0,\quad 1\le j\le p-1
\]
and
\[
\hat{u}_j\hat{v}_{-j} =
\lim_{m\rightarrow\infty}\widehat{x_j^{(m)}}\widehat{y_{-j}^{(m)}}=0,
\quad 1\le j\le p-1.
\]
Moreover, since $x_0^{(m)}=y_0^{(m)}=1$, we also have
$u_0v_0=0$ and hence by (\ref{2.8}) also
$\hat{u}_0\hat{v}_0=0$. We have thus proved that
\[
\supp(u) \cap \supp(v) = \emptyset \qquad\textrm{and}\qquad
\supp(\hat{u}) \cap (-\supp(\hat{v}) = \emptyset.
\]
However $u,v$ are non-zero, because $\|u\|_2=\|v\|_2=1$, so as 
in the proof of Theorem 3.5, this contradicts Proposition 3.3.
Therefore $\varphi:\mathbb{C}^{2p-2}\rightarrow\mathbb{C}^{2p-2}$
is a proper holomorphic function.
\end{proof}

\begin{lemma}\label{l5.2}
Let $\varphi: \mathbb{C}^{2p-2}\rightarrow\mathbb{C}^{2p-2}$ be
the proper holomorphic function defined in lemma 5.1. Put
\[
\mathbb{Z}_p^* = \mathbb{Z}_p\setminus\{0\} = \{1,2,\ldots,p-1\}.
\]
\begin{itemize}
\item[(i)] Assume $z=(x_1,\ldots,x_{p-1},y_1,\ldots,y_{p-1})$ is
a solution to $\varphi(z)=0$, and put
\[
x=(1,x_1,\ldots,x_{p-1}),\qquad y=(1,y_1,\ldots,y_{p-1}).
\]
Then there is a unique pair $(K,L)$
of subsets $K,L\subseteq\mathbb{Z}_p^*$ satisfying $|K|+|L|=p-1$,
such that
\begin{align}
\supp(x)=L\cup\{0\} &,\quad \supp(\hat{x})=K\cup\{0\}\label{5.5}\\
\supp(y)=\mathbb{Z}_p\setminus L &,\quad -\supp(\hat{y})=
\mathbb{Z}_p\setminus K\label{5.6}.
\end{align}
\item[(ii)] Conversely if $K,L\subseteq\mathbb{Z}_p^*$ satisfy
$|K|+|L|=p-1$, then there exists exactly one solution
$(x,y)\in\mathbb{C}^p\times\mathbb{C}^p$ to (\ref{5.5}) and
(\ref{5.6}) of the form $x=(1,x_1,\ldots,x_{p-1})$,
$y=(1,y_1,\ldots,y_{p-1})$ and for this solution,
$z=(x_1,\ldots,x_{p-1},y_1,\ldots,y_{p-1})\in\mathbb{C}^{2p-2}$
satisfies $\varphi(z)=0$.
\item[(iii)] The number of distinct zeros for $\varphi$ is equal
to $\left(\begin{smallmatrix}2p-2\\ p-1\end{smallmatrix}\right)$.
\end{itemize}
\end{lemma}

\begin{proof}
(i): Assume that $\varphi(z)=0$ for
$z=(x_1,\ldots,x_{p-1},y_1,\ldots,y_{p-1})\in\mathbb{C}^{2p-2}$,
and define $x,y\in\mathbb{C}^p$ as in (i). Then by the definition
of $\varphi$,
\begin{equation}\label{5.7}
x_jy_j=0,\quad \hat{x}_j\hat{y}_{-j}=0,\quad\textrm{for}
\quad 1\le j\le p-1.
\end{equation}
Moreover $x_0y_0=1$, so by (\ref{5.7}) and (\ref{2.8}) also
$\hat{x}_0\hat{y}_0=1$. Therefore
\begin{align*}
\supp(x)\cap\supp(y) &= \{0\}\\
\supp(\hat{x})\cap(-\supp(\hat{y})) &= \{0\}.
\end{align*}
Hence, there are unique subsets $K$, $K'$, $L$, $L'$ of
$\mathbb{Z}_p^*$ such that
\begin{align}
\supp(x) = L\cup\{0\}&, \quad \supp(\hat{x})=K\cup\{0\}\label{5.8}\\
\supp(y) = L'\cup\{0\}&,  \quad -\supp(\hat{y})=K'\cup\{0\}.\label{5.9}
\end{align}
Moreover $K\cap K'=\emptyset$ and $L\cap L'=\emptyset$. In particular
\begin{equation}\label{5.10}
|K|+|K'|\le p-1\quad\textrm{and}\quad |L|+|L'|\le p-1.
\end{equation}
By Proposition 3.3
\begin{align}
|K|+|L| &=|\supp(x)|+|\supp(\hat{x})|-2 \ge p-1\label{5.11}\\
|K'| + |L'| &= |\supp(y)|+|\supp(\hat{y})|-2\ge p-1.\label{5.12}
\end{align}
Hence, equality must hold in the 4 inequalities in (\ref{5.10}),
(\ref{5.11}) and (\ref{5.12}). In particular $|K|+|L|=p-1$ and 
$K'=\mathbb{Z}_p^*\setminus K$, $L'=\mathbb{Z}_p^*\setminus L$.
This proves (\ref{5.5}) and (\ref{5.6}), and the uniqueness of
$K$ and $L$ is clear.

(ii): Let $K,L\subseteq Z_p^*$ be such that $|K|+|L|=p-1$.
Put $K'=\mathbb{Z}_p^*\setminus K$,
$L'=\mathbb{Z}_p^*\setminus L$. Then (\ref{5.6}) can be written as
\begin{equation}\label{5.13}
\supp(y)=L'\cup\{0\},\quad -\supp(\hat{y})=K'\cup\{0\}.
\end{equation}
Moreover
\begin{equation}\label{5.14}
|K'|=|L|, |L'|=|K|.
\end{equation}
Assume first that $|K|\ge 1$ and $|L|\ge 1$. Then by Chebotar\"ev's
Theorem (Theorem 2.1), the submatrices $(F_p)_{K'\times L}$ and
$(F_p)_{K\times L'}$ of
\[
F_p=(\frac{1}{\sqrt{p}}\e^{\I 2\pi kl/p})_{j,k=0,\ldots,p-1}
\]
have non-zero determinants. We claim that (\ref{5.5}) and (\ref{5.6})
have a unique solution $(x,y)$ of the form
$x=(1,x_1,\ldots,x_{p-1})$, $y=(1,y_1,\ldots,y_{p-1})$ and that
this solution is given by
\begin{equation}\label{5.15}
\left\{\begin{array}{rcl}
\dps (x_l)_{l\in L} &=& \dps -\frac{1}{\sqrt{p}}
\big[(F_p)_{K'\times L}\big]^{-1}(1)_{k\in K'}\\
\dps x_l &=& 0\ \textrm{for}\ l\in L'
\end{array}\right.
\end{equation}
and
\begin{equation}\label{5.16}
\left\{\begin{array}{rcl}
\dps (y_l)_{l\in L'} &=& \dps -\frac{1}{\sqrt{p}}
\big[(\overline{F}_p)_{K\times L'}\big]^{-1}(1)_{k\in K}\\
\dps y_l &=& 0\ \textrm{for}\ l\in L
\end{array}\right.
\end{equation}
where $(1)_{k\in K}$ (resp. $(1)_{k\in K'}$) is the column 
vector with coordinates indexed by $K$ (resp. $K'$) and all 
entries equal to 1. Moreover $\overline{F}_p$ is the complex
conjugate of $F_p$.

To prove this claim, observe first that (\ref{5.5}) is equivalent to
\begin{equation}\label{5.17}
\supp(x)\subseteq L\cup\{0\},\quad \supp(\hat{x})\subseteq K\cup\{0\}
\end{equation}
because if one of the inclusions in (\ref{5.17}) is proper, then
\[
|\supp(x)|+|\supp(\hat{x})| < |K|+|L|+2 = p+1
\]
which contradicts Proposition 3.3. Moreover
$x=(1,x_1,\ldots,x_{p-1})$ satisfies (\ref{5.17}) if and only
if $x_l=0$ for $l\in L'$ and
\[
\frac{1}{\sqrt{p}}+\frac{1}{\sqrt{p}}\sum_{l\in L}\e^{\I 2\pi kl/p}x_l =
0,\quad k\in K'.
\]
The latter formula can be rewritten as
\[
(F_p)_{K'\times L}(x_l)_{l\in L} = -\frac{1}{\sqrt{p}}(1)_{k\in K'}
\]
which is equivalent to (\ref{5.15}). Similarly one gets that for
$y=(1,y_1,\ldots,y_{p-1})$, (\ref{5.6}) is equivalent to
\[
\supp(y) \subseteq L'\cup\{0\},\quad -\supp(\hat{y})\subseteq K'\cup\{0\}
\]
which is equivalent to $y_l=0$ for $l\in L$ and
\[
\frac{1}{\sqrt{p}}+\frac{1}{\sqrt{p}}\sum_{l\in L'}\e^{-\I 2\pi kl}y_l=0,\quad k\in K,
\]
and this is equivalent to (\ref{5.16}). Finally if $|K|=0$, then
$K=L'=\emptyset$ and $K'=L=\mathbb{Z}_p^*$. In this case, it is
elementary to check that the pair $x=(1,1,\ldots 1)$,
$y=(1,0,\ldots, 0)$ is the unique solution to (\ref{5.5}) and
(\ref{5.6}). Similarly, if $|L|=0$, the pair $x=(1,0,\ldots,0)$,
$y=(1,1,\ldots,1)$ is the unique solution to (\ref{5.5}) and
(\ref{5.6}).

Note finally, that if $(x,y)$ is a solution to (\ref{5.5}) and 
(\ref{5.6}) of the form $x=(1,x_1,\ldots,x_{p-1})$,
$y=(1,y_1,\ldots,y_{p-1})$, then
$z=(x_1,\ldots,x_{p-1},y_1,\ldots,y_{p-1})$ is a zero for
$\varphi$, because $\supp(x)\cap\supp(y)=\{0\}$ and
$\supp(\hat{x})\cap(-\supp(\hat{y}))=\{0\}$. This proves (ii).

(iii): By (i) and (ii) there is a one-to-one correspondence between
the zeros of $\varphi$ and pairs $(K,L)$ of subsets
$\mathbb{Z}_p^*$ satisfying $|K|+|L|=p-1$. The number of such
pairs is
\[
\sum_{j=0}^{p-1}\binom{p-1}{j}\binom{p-1}{p-1-j}=
\binom{2p-2}{p-1},
\]
which proves (iii).
\end{proof}

\begin{theorem}\label{t5.3}
The map $\varphi: \mathbb{C}^{2p-2}\rightarrow \mathbb{C}^{2p-2}$ 
defined in lemma 5.1 is a proper holomorphic function of multiplicity
$\left(\begin{smallmatrix}2p-2\\ p-1\end{smallmatrix}\right)$.
In particular the number of solutions
$(x_1,\ldots,x_{p-1},y_1,\ldots,y_{p-1})$ to (\ref{2.9})
counted with multiplicity is equal to
$\left(\begin{smallmatrix}2p-2\\ p-1\end{smallmatrix}\right)$.
\end{theorem}

\begin{proof}
By Theorem \ref{t4.8} it is sufficient to prove that for some
$w\in\mathbb{C}$, the number of solutions to $\varphi(z)=w$
counted with multiplicity is equal to
$\left(\begin{smallmatrix}2p-2\\ p-1\end{smallmatrix}\right)$.
Put now $w=0$. From lemma \ref{l5.2} we know that $\varphi$ 
has exactly $\left(\begin{smallmatrix}2p-2\\ p-1\end{smallmatrix}\right)$
distinct zeros. Hence we just have to show, that all the zeros
have multiplicity 1, or equivalently the Jacobian
$J(z)=\mathrm{det}(\varphi'(z))$ is non-zero whenever $\varphi(z)=0$.

Let $z=(x_1,\ldots,x_{p-1},y_1,\ldots,y_{p-1})$ be a zero for
$\varphi$, put $x=(1,x_1,\ldots,x_{p-1})$, $y=(1,y_1,\ldots,y_{p-1})$
and let $K,L\subseteq\mathbb{Z}_p^*$ be the corresponding sets as
in lemma \ref{l5.2}. Then $|K|+|L|=p-1$, and with
$K'=\mathbb{Z}_p^*\setminus K$, $L'=\mathbb{Z}_p^*\setminus L$,
(\ref{5.5}) and (\ref{5.6}) can be written
\begin{eqnarray}
\mathrm{supp}(x)= L\cup\{0\}, & \mathrm{supp}(\hat{x})=
K\cup\{0\}\label{5.18}\\
\mathrm{supp}(y)= L'\cup\{0\}, & -\mathrm{supp}(\hat{y})=
K'\cup\{0\}.\label{5.19}
\end{eqnarray}
In order to determine $\varphi'(z)$ we compute $\varphi(z+h)$ 
for $h=(f_1,\ldots,f_{p-1},g_1,\ldots,g_{p-1})\in\mathbb{C}^{2p-2}$.
Put
\[
f=(0,f_1,\ldots,f_{p-1}),\quad g=(0,g_1,\ldots,g_{p-1}).
\]
Then
\begin{align*}
\varphi(z+h)_j &= (x_j+f_j)(y_j+g_j),\quad 1\le j\le p-1\\
\varphi(z+h)_{p-1+j} &=
(\hat{x}_j+\hat{f}_j)(\hat{y}_{-j}+\hat{g}_{-j}),\quad 1\le j\le p-1.
\end{align*}
Using $\|\hat{f}\|_2\|\hat{g}\|_2=\|f\|_2\|g\|_2\le \|h\|^2_2$, we get
\begin{align*}
\varphi(z+h)_j &=\varphi(z)_j+f_jy_j+x_jg_j+O(\|h\|_2^2)\\
\varphi(z+h)_{p-1+j} &=
\varphi(z)_{p-1+j}+\hat{f}_j\hat{y}_{-j}+\hat{x}_j\hat{g}_{-j}+O(\|h\|_2^2)
\end{align*}
in Landau's $O$-notation. Hence
\begin{eqnarray}
(\varphi'(z)h)_j &= y_jf_j+x_jg_j,\quad 1\le j\le p-1\label{5.20}\\
(\varphi'(x)h)_{p-1+j} &= \hat{y}_{-j}\hat{f}_j+\hat{x}_j\hat{g}_{-j},
\quad 1\le j\le p-1.\label{5.21}
\end{eqnarray}
To prove that $J(z)=\det(\varphi'(z))\ne 0$, we just have to show
that $\ker(\varphi'(z))=0$, i.e.
\[
\varphi'(z)h=0\ \Rightarrow\ h=0,\quad h\in\mathbb{C}^{2p-2}.
\]
By (\ref{5.18}) and (\ref{5.19}), the formulas (\ref{5.20}) and
(\ref{5.21}) can be written as
\[
(\varphi'(z)h)_j = \left\{\begin{array}{cc} x_jg_j, & j\in L\\ y_jf_j,
& j\in L'\end{array}\right.
\]
and
\[
(\varphi'(z)h)_{p-1+j}=\left\{\begin{array}{cc}\hat{x}_j\hat{g}_{-j},
& j\in K\\\hat{y}_{-j}\hat{f}_j, & j\in K'\end{array}\right..
\]
Hence, if $\varphi'(z)h=0$, then by (\ref{5.18}) and (\ref{5.19}),
\begin{align*}
g_j&=0\ (j\in L),\quad  f_j=0\ (j\in L'),\\
\hat{g}_{-j}&=0\ (j\in K),\quad \hat{f}_j = 0 (j\in K'),
\end{align*}
and since $f_0=g_0=0$ by the definition of $f$ and $g$, it follows that
\begin{align*}
\mathrm{supp}(f)\subseteq L, &\quad
\mathrm{supp}(\hat{f})\subseteq K\cup\{0\}\\
\mathrm{supp}(g)\subseteq L', &\quad
-\mathrm{supp}(\hat{g})\subseteq K'\cup\{0\}.
\end{align*}
Hence
\[
|\mathrm{supp}(f)|+|\mathrm{supp}(\hat{f})| \le |K|+|L|+1 = p
\]
and
\[
|\mathrm{supp}(g)|+|\mathrm{supp}(\hat{g})| \le |K'|+|L'|+1 = p
\]
By Proposition \ref{p3.3}, it now follows that $f=g=0$ and
hence $h=0$. Therefore $\ker(\varphi'(z))=0$, and
hence $J(z)\ne 0$.
\end{proof}

\begin{corollary}\label{col5.4}
Let $x',y'\in\mathbb{C}^{p-1}$, $x'=(x_1,\ldots,x_{p-1})$,
$y'=(y_1,\ldots,y_{p-1})$ and put $x=(1,x_1,\ldots,x_{p-1})$,
$y=(1,y_1,\ldots,y_{p-1})$. Then the function
$\psi:\mathbb{C}^{2p-2}\rightarrow\mathbb{C}^{2p-2}$ given by
the coordinate functions
\begin{align}
\psi_j(x',y') &= x_jy_j,\quad 1\le j\le p-1\label{e5.22}\\
\psi_{p-1+j}(x',y') &= \sum_{m=0}^{p-1}x_{j+m}y_m,\quad 1\le j\le p-1\label{e5.23}
\end{align}
is a proper holomorphic function of multiplicity 
$\left(\begin{smallmatrix}2p-2\\ p-1\end{smallmatrix}\right)$.
In particular the number of solutions

\noindent
$(x_1,\ldots,x_{p-1},y_1,\ldots,y_{p-1})$ to (\ref{2.4})
counted with multiplicity is equal to
$\left(\begin{smallmatrix}2p-2\\ p-1\end{smallmatrix}\right)$.
\end{corollary}

\begin{proof}
Let $\varphi:\mathbb{C}^{2p-2}\rightarrow\mathbb{C}^{2p-2}$ 
be as in lemma \ref{l5.1}. By Proposition \ref{p2.3}
\begin{equation}
\varphi = \Lambda\circ\psi	\label{5.22}
\end{equation}
where $\Lambda:\mathbb{C}^{2p-2}\to \mathbb{C}^{2p-2}$ is the
affine map given by
\begin{equation}\label{e5.25}
\Lambda(a_1, \ldots, a_{p-1}, c_1,\ldots, c_{p-1})=
(a_1, \ldots, a_{p-1}, b_1,\ldots, b_{p-1})
\end{equation}
where
\begin{equation}\label{e5.26}
b_j={1\over p} (1+\sum^{p-1}_{m=1} a_m +
\sum^{p-1}_{k=1} e^{i2\pi jk/p} c_k), \quad 1\le j\le p=1
\end{equation}
Moreover by Proposition \ref{p2.3}, $\Lambda$ is a bijection
and its inverse is given by (2.13) with $n=p$.
Hence by (\ref{5.22})
\[
\psi=\Lambda^{-1} \circ\varphi
\]
where $\Lambda$ and $\Lambda^{-1}$ are affine transformations
of $\mathbb{C}^{2p-2}$. Therefore it follows from Theorem \ref{t5.3},
that $\psi$ is a proper holomorphic function of
multiplicity 
$\left(\begin{smallmatrix}2p-2\\ p-1\end{smallmatrix}\right)$,
so by Theorem \ref{t4.8} the number of solutions
$(x_1, \ldots, x_{p-1}, y_1, \ldots, y_{p-1})$ to
(2.4) counted with multiplicity is 
$\left(\begin{smallmatrix}2p-2\\ p-1\end{smallmatrix}\right)$.
\end{proof}

\section{The numbers of cyclic $p$-roots on $x$-level
and $z$-level} %
Throughout this section $p$ is again a prime
number. We will show that the numbers
of solutions to (2.3) and (2.1) counted
with multiplicity are both equal to
$\left(\begin{smallmatrix}2p-2\\ p-1\end{smallmatrix}\right)$.
In the case of (2.3), we consider $x_0$ as the
fixed number 1, so the problem has $p-1$ variables
$x_1, \ldots, x_{p-1}$.

\begin{lemma}\label{l6.1}
Put $a_0=x_0=1$ and define for $a=(a_1, \ldots, a_{p-1})\in
(\mathbb{C}^*)^{p-1}$
a map $\sigma_a:(\mathbb{C}^*)^{p-1}\to \mathbb{C}^{p-1}$ by
\[
\sigma_a (x_1,\ldots, x_{p-1})_j =
\sum^{p-1}_{m=0} a_m \frac{x_{m+j}}{x_m}, \quad
1\le j\le p-1
\]
Then $\sigma_a$ is a proper holomorphic function, and
the multiplicity of $\sigma_a$ is independent of
$a\in (\mathbb{C}^*)^{p-1}$.
\end{lemma}

\begin{proof}
Let $a\in (\mathbb{C}^*)^{p-1}$. Then $\sigma_a$ is clearly
holomorphic. To prove that $\sigma_a$ is proper,
we let $K\subseteq \mathbb{C}^{p-1}$ be compact.
Put $a_0=x_0=y_0=1$ and let $\psi$ be the holomorphic map
defined in Corollary 5.4. Since $\psi$ is proper, the set
\[
L_a = \psi^{-1} (\{a\} \times K)
\]
is compact. Moreover $L_a$ is the set of
$(x', y')=(x_1, \ldots, x_{p-1}, y_1, \ldots, y_{p-1})\in
\mathbb{C}^{2p-2}$
for which
\[
x_j y_j=a_j,\quad 1\le j\le p-1
\]
and
\[
\left(\sum^{p-1}_{m=0} x_{j+m} y_m\right)^{p-1}_{j=1}\in K
\]
Since $a_j\not= 0$ $(1\le j\le p-1)$, $L_a$ can be expressed as
the set of
\[
\big(x_1, \ldots, x_{p-1}, \frac{a_1}{x_1}, \ldots,
\frac{a_{p-1}}{x_{p-1}}\big)\in \mathbb{C}^{2p-2}
\]
for which $(x_1, \ldots, x_{p-1})\in (\mathbb{C^*})^{p-1}$ and
\[
\left(\sum^{p-1}_{m=0} a_m \frac{x_{j+m}}{x_m}\right)^{p-1}_{j=1}
\in K
\]
Hence $\sigma^{-1}_{a}(K)=\pi(L_a)$, where
$\pi:\mathbb{C}^{2p-2}\to \mathbb{C}^{p-1}$
is the map that takes out the first ${p-1}$ 
coordinates of an element in $\mathbb{C}^{2p-2}$.
Therefore $\sigma^{-1}_{a}(K)$ is compact, and we have proved
that $\sigma_a$ is proper.

Note that $(\mathbb{C}^*)^{p-1}$ is a connected open set
in $\mathbb{C}^{p-1}$. In order to prove that $a\to m(\sigma_a)$
is a constant function on $(\mathbb{C}^*)^{p-1}$, it is
therefore sufficient to prove that for
every $a_0\in (\mathbb{C}^*)^{p-1}$, $m(\sigma_a)$ is constant in
a ball $U=B(a_0, \varepsilon)$, where $\varepsilon>0$ is chosen such
that $\overline{U}\subseteq (\mathbb{C}^*)^{p-1}$. Put now
\[
M=\max \{\|a\|_2 \mid a\in \overline{U}\}.
\]
Since the map $\psi: \mathbb{C}^{2p-2}\to
\mathbb{C}^{2p-2}$ is proper,
we can choose ${R}>0$, such that
\begin{equation}\label{e6.1}
\| \psi(z)\|_2 \ge(M^2+1)^{1/2},\quad
{\text{when}}\quad  \| z\|_2\ge R
\end{equation}
Applying (6.1) to
\[
z=\Big(x_1, \ldots, x_{p-1}, \frac{a_1}{x_1}, \ldots,
\frac{a_{p-1}}{x_{p-1}}\Big)
\]
for $x'=(x_1, \ldots, x_{p-1})\in 
(\mathbb{C}^*)^{p-1}$, we get that
\[
\| a\|^2_2 +\| \sigma_a(x')\|^2_2=\| \psi (z)\|^2_2
\ge M^2+1
\]
when
\[
\| (x_1, \ldots, x_{p-1})\|^2_2 +\|\Big(\frac{a_1}{x_1},
\ldots, \frac{a_{p-1}}{x_{p-1}}\Big) \|^2_2 \ge R^2
\]
and since $\|a\|_2\le M$ for $a\in \overline{U}$ it follows
that
\begin{equation}\label{e6.2}
\begin{aligned}
\| \sigma_a(x')\|_2 & \ge 1,\quad {\text{when}}\quad
a\in\overline{U}\quad {\text{and}}\\
\| (x_1,\ldots, x_{p-1})\|_2 &\ge R\quad {\text{or}}\quad
\|\Big(\frac{a_1}{x_1},\ldots,\frac{a_{p-1}}{x_{p-1}}\Big) \|_2 \ge R
\end{aligned}
\end{equation}
Put
\[
D=\{(x_1, \ldots, x_{p-1}) \in\{ (\mathbb{C}^*)^n
\mid \frac{c}{R}<x_j <R\}
\]
where $c=\min\{|a_j| \mid a\in \overline{U},\ j=1,\ldots,
p-1\}>0$.
By replacing $R$ with a larger number, we
can assume that $\frac{c}{R}<R$. Then $\overline{D}$ is a
non-empty compact subset of $(\mathbb{C}^*)^{p-1}$ and
its boundary $\partial D$ has $2^{p-1}$ smooth
components. By (6.2) all the zeros of
$\sigma_a$ are in $D$, when $a\in \overline{U}$. Let
$a\in \overline{U}$.
Since $\overline{U}$ is convex, all the functions
\[
(1-t)\sigma_{a_0}+t\sigma_a,\quad 0\le t\le 1
\]
are of the form $\sigma_{a'}$ for an $a'\in\overline{U}$,
namely $a'=(1-t)a_0+t_a$. Hence by
applying Rouch\'es Theorem (Theorem 4.6)
to $f=\sigma_{a_0}$ and $g=\sigma_a-\sigma_{a_0}$, we get
that $\sigma_{a_0}$ and $\sigma_a$ have the same number
of zeros in $D$ counted with multiplicity,
and since neither $\sigma_{a_0}$ nor $\sigma_a$ has zeros
in $(\mathbb{C}^*)^{p-1}\setminus D$, it follows that $\sigma_{a_0}$
and $\sigma_a$ have the same number of zeros in
$(\mathbb{C}^*)^{p-1}$ counted with multiplicity. Therefore
by Theorem 4.8, $m(\varphi_a)=m(\varphi_{a_0})$ for all 
$a\in \overline{U}$. Hence we have
proved that $m(\varphi_a)$ is a constant function on
$(\mathbb{C}^*)^{p-1}$.
\end{proof}

\begin{theorem}\label{t6.2}
Put $x_0=1$, and let $\sigma:(\mathbb{C}^*)^{p-1}\to
\mathbb{C}^{p-1}$ be
the function defined by
\[
\sigma(x_1, \ldots, x_{p-1})_j=\sum^{p-1}_{m=0}
\frac{x_{m+j}}{x_m}\quad 1\le j\le p-1
\]
Then $\sigma$ is a proper holomorphic function of
multiplicity 
$\left(\begin{smallmatrix}2p-2\\ p-1\end{smallmatrix}\right)$.
In particular there are
$\left(\begin{smallmatrix}2p-2\\ p-1\end{smallmatrix}\right)$
cyclic $p$-roots on $x$-level counted with multiplicity.
\end{theorem}

\begin{proof}
Let $\psi: \mathbb{C}^{2p-2}\to \mathbb{C}^{2p-2}$
be the holomorphic function defined in Theorem 5.4.
Then $\psi$ is proper and
has multiplicity $m(\psi)=
\left(\begin{smallmatrix}2p-2\\ p-1\end{smallmatrix}\right)$.
Let $N=\psi(M)$ denote the set of critical values for
$\psi$. Then by Theorem 4.1, and Remark 4.2, $N$ is
a closed set, and
\[
m_{4p-4} ({N})=0
\]
where $m_{4p-4}$ is the Lebesgue measure in
$\mathbb{C}^{2p-2}\simeq \mathbb{R}^{4p-4}$.
By Theorem 4.1 the number of
district solutions $z\in \mathbb{C}^{2p-2}$ to
\[
\varphi(z)=w
\]
is $m(\varphi)$ for every $w=(a, c)\in(\mathbb{C}^{p-1}\times
\mathbb{C}^{p-1}) \setminus N$. Since
$m_{4p-4} (N)=(m_{2p-2} \times m_{2p-2})(N),$
where $m_{2p-2}$ is the Lebesgue measure in 
$\mathbb{C}^{p-1}$, it follows that
\[
0=m_{4p-4} (N)=\int_{\mathbb{R}^{2p-2}} m_{2p-2}(N_a)
dm_{2p-2} (a)
\]
where $N_a=\{c\in \mathbb{C}^{p-1}\mid (a,c)\in N\}$
(see e.g. [Ru 2, Sect. 8].) Hence the set
\[
N'=\{a\in \mathbb{C}^{p-1}\mid m_{2p-2} (N_a)\not= 0\}
\]
is a $m_{2p-2}$--null set in $\mathbb{C}^{p-1}$. Moreover
for all $a\in\ \mathbb{C}^{p-1}\setminus N'$, the
number of district solutions to
\begin{equation}\label{e6.3}
\psi(z)=(a, c)
\end{equation}
is exactly $m(\psi)$ for all $c\in \mathbb{C}^{p-1}$
outside  the Lebesgue null set $N_a$. If $a\in 
(\mathbb{C}^*)^{p-1}$ we have from the proof of lemma 6.1.
that the solution (6.3) are precisely the elements
in $(\mathbb{C}^*)^{2p-2}$ of the form
\[
\Big(x_1, x_2, \ldots, x_{p-1}, \frac{a_1}{x_1},\ldots,
\frac{a_{p-1}}{x_{p-1}}\Big)
\]
for which $\sigma_a(x_1, \ldots, x_{p-1})=c$. Hence for
$a\in (\mathbb{C}^*)^{p-1}\setminus N'$, the number of 
distinct solutions to $\sigma_a(x')=c$ is equal to
$m(\psi)$ for Lebseque almost all 
$c\in \mathbb{C}^{p-1}$. Therefore by
Theorem 4.1 and Remark 4.2, the multiplicity
$m(\sigma_a)$ of $\sigma_a$ is equal to $m(\psi)$ for
all $a\in (\mathbb{C}^*)^{p-1}\setminus N'$.
But since $a\to m(\sigma_a)$ is a constant function
on $(\mathbb{C}^*)^{p-1}$ by lemma 6.1,
it follows that $m(\sigma_a)=m(\psi)$ for all
$a\in (\mathbb{C}^*)^{p-1}$. Putting 
$a=(1, \ldots, 1)$, we get in particular, that 
$m(\sigma)=m(\psi)=
\left(\begin{smallmatrix}2p-2\\ p-1\end{smallmatrix}\right)$.
Thus by Theorem 4.8, the number of 
solution $(x_1, \ldots, x_{p-1})$ to (2.3) counted with
multiplicity is equal to
$\left(\begin{smallmatrix}2p-2\\ p-1\end{smallmatrix}\right)$
where $n=p$.
\end{proof}

\begin{lemma}\label{l6.3}
Put $x_0=1$ and let $h:(\mathbb{C}^*)^p\to
(\mathbb{C}^*)^{p}$ be the
function given by
\begin{equation}\label{e6.4}
h(x_1, \ldots, x_{p-1}, \alpha)=\Big(\frac{\alpha x_1}{x_0},
\frac{\alpha x_2}{x_1}, \ldots, \frac{\alpha x_0}{x_{p-1}}\Big)
\end{equation}

Then $h$ is proper, and for every 
$(z_0, \ldots, z_{p-1})\in (\mathbb{C}^*)^p$
there are exactly $p$ distinct solutions 
in $(\mathbb{C}^*)^p$ to the equation 
\begin{equation}\label{e6.5}
h(x_1, \ldots, x_{p-1}, \alpha)=
(z_0, \ldots, z_{p-1})
\end{equation}
\end{lemma}

\begin{proof}
We start by solving (6.5) w.r.t.
$(x_1, \ldots, x_{p-1}, \alpha)$. By (6.4)
\begin{equation}\label{e6.6}
z_0 z_1\cdot \ldots \cdot z_{p-1}=\alpha^p
\end{equation}
Hence $\alpha$ is one of the $p$ distinct $p$'th roots of
$z_0 z_1\cdot \ldots\cdot z_{p-1}$. For each such $\alpha$,
there is a unique solution to (6.5) given by
\begin{equation}\label{e6.7}
x_1=\frac{z_0}{\alpha},\ x_2=\frac{z_0 z_1}{\alpha^2}, \ldots,
x_{p-1}=\frac{z_0 z_1\cdot \ldots \cdot z_{p-2}}{\alpha^{p-1}}
\end{equation}
Hence (6.5) has exactly $p$ distinct solutions.
Let $K\subseteq (\mathbb{C}^*)^p$ be compact. Then there exists
$R>0$, such that
\[
K\subseteq \{z\in (\mathbb{C}^*)^p \mid
\frac{1}{R} \le \mid z_j\mid \le R,\ 0\le j\le p-1\}.
\]
>From (6.6) and (6.7) it now follows that
$h^{-1} (K)$ is relatively compact in
$(\mathbb{C}^*)^p$, which
by the continuity of $h$ implies that
$h^{-1}(K)$ is compact.
Hence $h$ is proper.
\end{proof}

\begin{theorem}\label{t6.4}
Let $\rho: \mathbb{(C^*)}^{p}\to \mathbb{C}^{p-1}\times
\mathbb{C^*}$ be the function given
by
\begin{align*}
\rho_1(z) &= z_0 + z_1 + \ldots + z_{p-1}\\
\rho_2(z) &= z_0z_1 + z_1z_2 + \ldots + z_{p-1} z_0\\
\vdots &\\
\rho_{p-1}(z) &= z_0z_1\cdot \ldots \cdot z_{p-2} + \ldots + z_{p-1} z_0
\cdot \ldots \cdot z_{p-3}\\
\rho_p (z) & =z_0z_1\cdot \ldots\cdot z_{p-1}
\end{align*}
Then $\rho$ is a proper holomorphic function of
multiplicity 
$\left(\begin{smallmatrix}2p-2\\ p-1\end{smallmatrix}\right)$.
In particular, the numbers of cyclic $p$-roots on
$z$-level (i.e. the number of solutions to
(2.1) counted with multiplicity is equal to
$\left(\begin{smallmatrix}2p-2\\ p-1\end{smallmatrix}\right)$.
\end{theorem}

\begin{proof}
Consider the composed map $\rho\circ h: \mathbb{(C^*)}^{p}\to
\mathbb{C}^{p-1}\times \mathbb{C^*}$,
where $h$ is given by (6.4) with $x_0 =1$. Then
\[
(\rho\circ h)_j (x_1, \ldots, x_{p-1}, \alpha)=
\alpha^j \sum^{p-1}_{m=0} \frac{x_{m+j}}{x_m},\quad
1\le j\le p-1
\]
and
\[
(\rho\circ h)_p (x_1, \ldots, x_{p-1}, \alpha)=\alpha^p
\]
Let $\sigma:(\mathbb{C}^*)^{p-1}\to \mathbb{C}^{p-1}$
be the proper holomorphic map from Theorem 6.2. Then
for $x'=(x_1, \ldots, x_{p-1})\in\mathbb{(C^*)}^{p-1}$ and
$\alpha \in \mathbb{C^*}$
\begin{equation}\label{e6.8}
(\rho\circ h) (x', \alpha) =(\alpha \sigma_1 (x'), \ldots,
\alpha^{p-1} \sigma_{p-1} (x'), \alpha^p)
\end{equation}
Since $\sigma:(\mathbb{C}^*)^{p-1}\to \mathbb{C}^{p-1}$ is
proper, it is elementary to deduce from (6.8), that
$\rho\circ h$ is a proper map from $\mathbb{(C^*)}^p$ to
$\mathbb{C}^{p-1}\times \mathbb{C^*}$. Moreover since
$h$ maps $(\mathbb{C^*})^p$ into
$(\mathbb{C^*)}^p$, we have for every compact subset
$K$ of 
$\mathbb{C}^{p-1}\times \mathbb{C^*}$, that
\[
\rho^{-1}(K)=h(h^{-1} (\rho^{-1}(K))=h((\rho\circ h)^{-1}
(K)),
\]
which is compact by the properness of
$\rho \circ h$. Hence $\rho$ is proper.

We will prove that $m(\rho) =m(\sigma)$ by
computing the multiplicity of $g\circ h$ in two
ways: By Theorem 4.1 and Remark 4.2,
there exists a Lebesgue nullset $N_0\subseteq \mathbb{C}^{p-1}$
such that for all $w\in \mathbb{C}^{p-1}\setminus N_0$ the
equation $\sigma_p(x')=w$ has $m(\sigma)$ district solutions
in $\mathbb{(C^*)}^{p-1}$. For $(x', \alpha)\in (\mathbb{C}^{p-1}
\setminus N_0) \times \mathbb{C^*}$,
$(\rho\circ h)(x', \alpha)=w$ if and
only if
\begin{equation}\label{e6.9}
\alpha^p = w_p
\end{equation}
and
\begin{equation}\label{e6.10}
\sigma(x')=\Big(\frac{1}{\alpha} w_1, \ldots, \frac{1}{\alpha^{p-1}}
w_{p-1}\Big)
\end{equation}
Since (6.9) has exactly $p$ distinct
solutions, it follows that $(\rho\circ h)(x', \alpha)=w$
has exactly $p m(\sigma)$ distinct solution
for all such $w$. The complement of
$(\mathbb{C}^{p-1}\setminus N_0)\times \mathbb{C^*}$ in
$\mathbb{C}^{p-1}\times \mathbb{C^*}$ is 
$N_0 \times \mathbb{C^*}$ 
which is a null set w.r.t. the Lebesgue
measure in $\mathbb{C}^p$. Hence by Theorem 4.1
and Remark 4.2, $m(\rho\circ h)= p m (\sigma)$.

By the definition of $m(\rho)$, there exists
a Lebesgue null set $N$ in $\mathbb{C}^{p-1} \times \mathbb{C^*}$,
such that for all $w\in \mathbb{C}^{p-1}\times \mathbb{C^*}
\setminus N$, the number of
distinct solutions $z\in \mathbb{(C^*)}^p$ to 
$\rho(z)=w$ is equal to $m(\rho)$. By lemma 6.3
we then get that the number of distinct
solutions $u\in \mathbb{(C^*)}^p$ to $\rho(h(u))=w$ 
is equal to $pm(\rho)$. Since $N$ is a Lebesgue nullset
it follows that $m(\rho\circ h)=p\cdot m(\rho)$. Hence
\[
m(\rho)=\frac{1}{p} m(\rho\circ h)=m(\sigma)=
\left(\begin{smallmatrix}2p-2\\ p-1\end{smallmatrix}\right).
\]
By Theorem 4.8 the number of solutions
to (2.1) with $n=p$ counted with multiplicity is equal
to $\left(\begin{smallmatrix}2p-2\\ p-1\end{smallmatrix}\right).$
\end{proof}

\section{Cyclic $p$-roots of simple index $k$}
Let $p$ be a prime number and let $k\in\mathbb{N}$ be a number that divides $p-1$. Since the group $(\mathbb{Z}_p^*,\cdot)$ is cyclic, it has a unique subgroup $G_0$ of index $k$, namely
\[
G_0=\{ h^k | h\in\mathbb{Z}_p^*\}.
\]
Moreover, if $g\in\mathbb{Z}_p^*$ is a generator for $\mathbb{Z}_p^*$, then
\[
G_l=g^l G_0,\qquad 1\le l\le k-1
\]
are the $k-1$ non-trivial cosets of $G_0$ in $\mathbb{Z}_p^*$. Following the notation of \cite{bh}, a cyclic $p$-root $z=(z_0,z,\ldots,z_{p-1})$ has simple index $k$ if the corresponding cyclic $p$-roots on $x$-level
\[
x=(1,z_0,z_0z_1,\ldots,z_0z_i\cdot\ldots\cdot z_{p-2})
\]
is of the form
\begin{equation}\label{7.1}
\left\{\begin{array}{l}
x_0=1\\
x_i=c_l,\quad \textrm{if}\quad i\in G_l,\quad 1\le i\le p-1,
\end{array}\right.
\end{equation}
where $(c_0,c_1,\ldots,c_{k-1})\in (\mathbb{C^*})^k$. These special cyclic $p$-roots where introduced by Bj\" orck in \cite{bj} under a slightly different name (cyclic $p$-roots of simple preindex $k$). It was shown in \cite{bj}, that if $x=(1,x_1,x_2,\ldots,x_{p-1})$ has the form (\ref{7.1}), then the equations (\ref{2.3}) can be reduced to the following set of $k$ rational equations in $c_0,\ldots,c_{k-1}$:
\begin{equation}\label{7.2}
c_a+\frac{1}{c_{a+m}}+\sum_{i,j=0}^{k-1} n_{ij}\frac{c_{a+j}}{c_{a+i}}=0\qquad (0\le a\le k-1)
\end{equation}
where indices are calculated modulo $k$. In (\ref{7.2}) the number $m$ is determined by $p-1\in G_m$ and $n_{ij}$ denote the number of $b\in G_i$ for which $b+1\in G_{i+1}$ $(0\le i,j\le k-1)$. The set of equations (\ref{7.2}) is independent of the choice of the generator $g$ for $\mathbb{Z}_p^*$ up to permutation of the variables and of the equations. The main result of this section is:

\begin{theorem}\label{thm7.1}
For every $k\in\mathbb{N}$ and for every prime number $p$ for which $k$ divides $p-1$, the function $\chi: (\mathbb{C}^*)^k\rightarrow \mathbb{C}^k$ given by
\[
\chi(c_0,\ldots,c_{k-1})_a = c_a+\frac{1}{c_{a+m}}+\sum_{i,j=0}^{k-1} n_{ij}\frac{c_{a+j}}{c_{a+i}} = 0,\qquad (0\le a \le k-1)
\]
is a proper holomorphic function of multiplicity $\left(\begin{smallmatrix}2k\\ k\end{smallmatrix}\right)$. In particular the number of solutions $(c_0,\ldots,c_{k-1})\in(\mathbb{C}^*)^k$ to \emph{(\ref{7.2})} counted with multiplicity is equal to $\left(\begin{smallmatrix}2k\\ k\end{smallmatrix}\right)$.
\end{theorem}
The proof of Theorem 7.1 relies on Proposition 7.3 and Proposition 7.4 below. We first introduce some notation: Let $n\in\mathbb{N}$ and let $F$ be a subspace of $\mathbb{C}^n$ of dimension $d\ge 1$. A subset $U\subseteq F$ is called a region in $F$ if it is non-empty, open and connected in the relative topology on $F$. By choosing a fixed basis for $F$, we can identify $F$ with $\mathbb{C}^d$, and thereby extend the defintion of holomorphic functions, proper holomorphic functions and their multiplicities to maps $\varphi: U\rightarrow V$, where $U$ and $V$ are two regions in $F$. Clearly these definitions are independent of the choice of a basis for $F$.

\begin{defi}\label{def7.2}
Let $E$ denote the set of $(x_i)_{i=1}^{p-1}\in\mathbb{C}^{p-1}$ for which the function $i\rightarrow x_i, i\in\mathbb{Z}_p^*=\{1,\ldots,p-1\}$ is constant on each of the cosets $G_0,\ldots,G_{k-1}$ of $G_0$.
\end{defi}

Note that $E$ is the $k$-dimensional subspace of $\mathbb{C}^{p-1}$, and the indicator functions $1_{G_0},\ldots,1_{G_{k-1}}$ given by
\begin{equation}\label{7.3}
(1_{G_l})_i = \left\{\begin{array}{ll}1 & i\in G_l\\ 0 & i \not\in G_l\end{array}\right.
\end{equation}
form a basis for $E$. Note also, that $E\times E$ is a subspace of $\mathbb{C}^{p-1}\times \mathbb{C}^{p-1}\simeq\mathbb{C}^{2p-2}$ of dimension $2k$.

\begin{prop}\label{prop7.3}
Let $\varphi,\psi:\mathbb{C}^{2p-2}\rightarrow\mathbb{C}^{2p-2}$ be the proper holomorphic functions defined in Lemma 5.1 and corollary 5.4. Then
\begin{itemize}
\item[(a)] $\varphi(E\times E) \subseteq E\times E$ and $\psi(E\times E)\subseteq E\times E$.
\item[(b)] The restrictions $\varphi_E$ and $\psi_E$ of $\varphi$ and $\psi$ to $E\times E$ are proper holomorphic functions.
\item[(c)] The multiplicities of $\varphi_E$ and $\psi_E$ are given by
\[
m(\varphi_E)= m(\psi_E) = \begin{pmatrix}2k\\ k\end{pmatrix}.
\]
\end{itemize}
\end{prop}

\begin{proof}
(a) Let $x'=(x_1,\ldots,x_{p-1})\in E$, $y'=(y_1,\ldots,y_{p-1})\in E$ and put
\[
x=(1,x_1,\ldots,x_{p-1})\qquad\textrm{and}\qquad y=(1,y_1,\ldots,y_{p-1}).
\]
To prove that $\varphi(E\times E)\subseteq E\times E$ and $\psi(E\times E)\subseteq E\times E$, it is by (\ref{5.1}), (\ref{5.2}), (\ref{e5.22}) and (\ref{e5.23}) sufficient to show that
\begin{eqnarray}
(x_j y_j)_{1\le j\le p-1}\in E\label{7.4}\\
(\hat{x}_j\hat{y}_{-j})_{1\le j\le p-1}\in E\label{7.5}\\
\Big(\sum_{m=0}^{p-1} x_{j+m}y_m\Big)_{1\le j\le p-1}\in E\label{7.6}
\end{eqnarray}
Note that (\ref{7.4}) follows immediately from the conditions $x'\in E$ and $y'\in E$. To prove (\ref{7.5}), note first that $G_0$ acts transitively on each of its cosets, i.e.
\[
G_l=\{ hj | h\in G_0\}\quad\textrm{for all}\quad j\in G_l.
\]
Hence
\begin{equation}\label{7.7}
E=\{ (x_j)_{j=1}^{p-1} | x_{hj} = x_j\quad\textrm{for all}\quad h\in G_0 \}
\end{equation}
where as usual indices are calculated modulo $p$. Let $h\in G_0$ and $0\le j\le p-1$. Then
\[
\hat{x}_{hj}=\frac{1}{\sqrt{p}}\bigg(\sum_{m=0}^{p-1}\e^{\I 2\pi jhm/p}x_m\bigg)
\]
Since $m\rightarrow hm$ is a bijetion of $\mathbb{Z}_p$ onto itself, we can replace $m$ by $h^{-1}m$ in the above summation ($h^{-1}$ is the inverse of $h$ in the group $G_0\subseteq \mathbb{Z}_p^*$). Hence
\begin{equation}\label{7.8}
\hat{x}_{hj}=\frac{1}{\sqrt{p}}\bigg(\sum_{m=0}^{p-1}\e^{\I 2\pi jm/p}x_{h^{-1}m}\bigg).
\end{equation}
Since $(x_1,\ldots,x_{p-1})\in E$ and $h^{-1}0=0$ we have $x_{h^{-1}m}=x_m$ for $0\le m\le p-1$ and therefore $\hat{x}_{hj}=\hat{x}_j$, $j\in\mathbb{Z}_p$. In the same way we get $\hat{y}_{-hj}=\hat{y}_{-j}$, $j\in\mathbb{Z}_p$. Hence (\ref{7.5}) follows from (\ref{7.7}).

To prove (\ref{7.6}), put $w=(w_1,\ldots,w_{p-1})$, where
\[
w_j=\sum_{m=0}^{p-1} x_{j+m}y_m,\qquad 1\le j\le p-1.
\]
Let $h\in G_0$. Then
\[
w_{hj} = \sum_{m=0}^{p-1} x_{h_j+m}y_m,\qquad 1\le j\le p-1.
\]
By replacing $m$ by $hm$ in the above summation, we get
\[
w_{hj} = \sum_{m=0}^{p-1} x_{h(j+m)}y_{hm}.
\]
Since $x',y',\in E$ and $h0=0$, it follows that
\[
w_{hj}=\sum_{m=0}^{p-1}x_{j+m}y_m=w_j.
\]
Hence by (\ref{7.7}), $w\in E$ which proves (\ref{7.6}).\\
\\
(b) It is clear that $\varphi_E$ and $\psi_E$ are holomorphic functions on $E\times E$. Let $K\subseteq E\times E$ be a compact set. Then
\[
(\varphi_E)^{-1}(K) = \varphi^{-1}(K)\cap (E\times E).
\]
Since $\varphi$ is proper, it follows that $\varphi_E$ is a proper holomorphic function of $E\times E$ into itself. The same argument shows that $\psi_E$ is proper.\\
\\
(c) Assume that $z=(x_1,\ldots,x_{p-1},y_1,\ldots,y_{p-1})\in E\times E$ is a solution to $\varphi(z)=0$, and put
\[
x=(1,x_1,\ldots,x_{p-1})\qquad\textrm{and}\qquad y=(1,y_1,\ldots,y_{p-1}).
\]
By Lemma 5.2 there exists a unique pair of subsets $K,L\subseteq\mathbb{Z}_p^*$ satisfying $|K|+|L|=p-1$ such that
\begin{eqnarray}
\textrm{supp}(x)=L\cup\{0\},\qquad \textrm{supp}(\hat{x})=K\cup\{0\} \label{7.9}\\
\textrm{supp}(y)=\mathbb{Z}_p\setminus L,\qquad -\textrm{supp}(\hat{y})=\mathbb{Z}_p\setminus K \label{7.10}
\end{eqnarray}
Let $h\in G_0$. Since $z\in E\times E$ we get from the proof of (a), that $x_{hj}=x_j$ and $\hat{x}_{hj}=\hat{x}_j$ for $1\le j\le p-1$. Hence the sets $K, L\in\mathbb{Z}_p^*$ are invariant under multiplication by all $h\in G_0$, which implies that $K$ and $L$ are disjoint unions of $G_0$-cosets, i.e
\begin{equation}\label{7.11}
K=\bigcup_{l\in I}G_l,\qquad L=\bigcup_{l\in I'}G_l
\end{equation}
where $I$ and $I'$ are finite subsets of $\{0,\ldots,k-1\}$. Moreover $|I|+|I'|=k$, because $|K|+|L|=p-1$ and each coset $G_l$ has $\frac{p-1}{k}$ elements.

Conversely, if $K,L$ are of the form (\ref{7.11}) for $I, I'\subseteq \{0,\ldots,k-1\}$ and $|I|+|I'|=K$, then by Lemma 5.2 (ii), there is precisely one element $(x,y)\in\mathbb{C}^p\times\mathbb{C}^p$ with $x_0=y_0=1$ for which (\ref{7.9}) and (\ref{7.10}) holds and for this pair $(x,y)$,
\[
z=(x_1,\ldots,x_{p-1},y_1,\ldots,y_{p-1})
\]
is a solution to $\varphi(z)=0$. We claim that $z\in E\times E$. To prove this, let $h\in G_0$ and define $(\tilde{x},\tilde{y})\in\mathbb{C}^p\times\mathbb{C}^p$ by
\[
\tilde{x}_j=x_{hj}\quad\textrm{and}\quad \tilde{y}_j=y_{hj},\qquad 0\le j\le p-1.
\]
Then $\tilde{x}_0=\tilde{y}_0=1$ and by the proof of (\ref{7.8})
\[
(\hat{\tilde{x}})_{j} = \hat{x}_{h^{-1}j}\quad\textrm{and}\quad (\hat{\tilde{y}})_{-j} = \hat{y}_{-h^{-1}j},\qquad 0\le j\le p-1.
\]
Since $h,h^{-1}\in G_0$ and since $K$ and $L$ are invariant under multiplication by elements from $G_0$, it follows that (\ref{7.9}) and (\ref{7.10}) are satisfied for the pair $(\tilde{x},\tilde{y})$ as well. Thus by the uniqueness of $(x,y)$ in Lemma 5.2 (ii), we have $\tilde{x}=x$ and $\tilde{y}=y$. Hence by (\ref{7.7}), $z\in E\times E$ as claimed.

Altogether, we have established a one-to-one correspondence between the zeros of $\varphi_E$ and the pairs of subsets $(I,I')$ of $\{0,\ldots,k-1\}$ for which $|I|+|I'|=k$. Hence $\varphi_E$ has exactly
\[
\sum_{l=0}^k\binom{k}{l}\binom{k}{k-l}=\binom{2k}{j}
\]
zeros. Let $z$ be a zero for $\varphi_E$. Then $z\in E\times E$ and $\varphi(z)=0$. By the proof of Theorem 5.3, $\ker\varphi'(z)=\{0\}$ and since $\varphi'_E$ is the restriction of $\varphi'(z)$ to $E\times E$ also $\ker\varphi'_E(z)=\{0\}$. Therefore all the zeros of $\varphi_E$ have multiplicity $1$. It now follows from Theorem 4.8, that $m(\varphi_E)=\binom{2k}{k}$.

From the proof of corollary 5.4 we know that $\psi=\Lambda^{-1}\circ\varphi$, where $\Lambda$ is the affine transformation of $\mathbb{C}^{2p-2}$ given by (\ref{e5.25}) and (\ref{e5.26}). It is elementary to check, that $\Lambda_E=\Lambda_{| E\times E}$ is an affine transformation of $E\times E$ onto itself. Hence $\psi_E=\Lambda^{-1}_E\circ\varphi_E$, and therefore $m(\psi_E)=m(\varphi_E)=\binom{2k}{k}$.
\end{proof}

\begin{prop}\label{prop7.4}
Let $\sigma : (\mathbb{C}^*)^{p-1}\rightarrow\mathbb{C}^{p-1}$ be the proper holomorphic map defined in Theorem 6.2, i.e.
\[
\sigma(x_1,\ldots,x_{p-1})_j = \sum_{m=0}^{p-1}\frac{x_{m+j}}{x_m},\qquad 1\le j\le p-1
\]
where $x_0=1$. Then the restriction $\sigma_E$ of $\sigma$ to $E_0=E\cap (C^*)^{p-1}$ is a proper holomorphic function of $E_0$ into $E$ with multiplicity $\binom{2k}{k}$.
\end{prop}

\begin{proof}
Note first that $E_0=E\cap(C^*)^{p-1}$ is an open, connected and dense subset of $E$. Put $a_0=1$ and define for $a\in E_0$
\[
\sigma_a(x_1,\ldots,x_{p-1})_j = \sum_{m=0}^{p-1} a_m\frac{x_{m+j}}{x_m},\quad 1\le j\le p-1
\]
as in lemma 6.1. It is clear from the proof of (\ref{7.6}) that $\sigma_a(E_0)\subseteq E$ for all $a\in E_0$. Let $\sigma_{a,E}$ denote the restriction of $\sigma_a$ to $E_0$. By lemma 6.1, $\sigma_a$ is a proper holomorphic map from $(\mathbb{C}^*)^{p-1}$ to $\mathbb{C}^{p-1}$. As in the proof of Proposition 7.4(b), it follows that $\sigma_a$ is a proper holomorphic map from $E_0$ to $E$. By simple modifications of the proofs of lemma 6.1 and Theorem 6.2 one gets first that the multiplicity of $\sigma_{a,E}$ is independent of $a\in E_0$ and next that $m(\sigma_{a,E})=m(\psi_E)$ for all $a\in E_0$. In particular
\[
m(\sigma_E)=m(\psi_E)=\binom{2k}{k}.
\]

\noindent\textbf{Proof of Theorem 7.1}. The function $\chi: (\mathbb{C}^*)^{k-1}\rightarrow\mathbb{C}^{k-1}$ defined in Theorem 7.1 is just the function $\sigma_E: E_0\rightarrow E$ written out in coordinates $(c_0,\ldots,c_{k-1})$ with respect to the basis $(1_{G_0},\ldots,1_{G_{k-1}})$ for $E$ defined by (\ref{7.3}) (cf. the derivation of the equations (\ref{7.2}) in \cite{bj}). Therefore Theorem 7.1 is an immediate consequence of Proposition 7.4 and Theorem 4.8.
\end{proof}

\begin{remark}\rm
(a) If $k=p-1$ all cyclic $p$-roots are of simple index $k$, and this special case of Theorem 7.1 is the same as Theorem 6.2.

(b) It follows from Theorem 7.1 that there are at most $\binom{2k}{k}$ distinct cyclic $p$-roots of simple index $k$ on $x$-level (or $z$-level). Moreover the number of cyclic $p$-roots of simple index $k$ on $x$-level (or $z$-level) counted with multiplicity is at least $\binom{2k}{k}$. However, for $k<p-1$, we have not been able to rule out the possibility that a cyclic $p$-root of simple index $k$ could have higher multiplicity with respect to the set of equations (\ref{2.3}) than with respect to the set of equations (\ref{7.2}).
\end{remark}

\bigskip
Uffe Haagerup\\
Department of Mathematics and Computer Science\\
University of Southern Denmark\\
Campusvej 55, DK-5230 Odense M\\
Denmark\\
\texttt{haagerup@imada.sdu.dk}

\end{document}